\documentclass[12pt,a4paper,reqno]{amsart}
\usepackage{geometry}
\usepackage[numbers]{natbib}
\usepackage{amsthm,amsmath,amssymb}
\newtheorem{theorem}{Theorem}
\newtheorem{corollary}[theorem]{Corollary}
\newtheorem{conjecture}[theorem]{Conjecture}
\newtheorem{proposition}[theorem]{Proposition}
\newtheorem{lemma}[theorem]{Lemma}

\theoremstyle{remark}
\newtheorem*{remark}{Remark}

\numberwithin{equation}{section}

\begin{document}

\title[Negative discrete Dirichlet $L$-function moments]{Negative discrete second moments of Dirichlet $L$-functions}

\author[A. Pearce-Crump]{Andrew Pearce-Crump}
\address{School of Mathematics, University of Bristol, Bristol, BS8 1UG, United Kingdom}
\email{andrew.pearce-crump@bristol.ac.uk}

\begin{abstract}
Let $\chi$ be a primitive Dirichlet character modulo $q>1$. Assuming the Generalised Riemann Hypothesis for $L(s,\chi)$ and that the non-trivial zeros $\rho=\tfrac12+i\gamma$ of $L(s,\chi)$ are simple, we prove lower bounds for the discrete moments $\sum_{0<\gamma\le T}|L'(\rho,\chi)|^{-2}$ and $\sum_{0<\gamma\le T}|L(2\rho,\chi^2)/L'(\rho,\chi)|^2$, uniformly in the conductor. The bounds capture the proportion $\beta/(1+\beta)$ of the conjectured asymptotics, where $\beta=\log T/\log qT$: this is one half whenever $\log q=o(\log T)$, recovering for fixed $q$ the Dirichlet analogues of theorems of Milinovich and Ng and of Sinha, and degrades to $1/(2+A)$ when $q=T^{A}$. We conjecture the true leading order asymptotics and their analogues when we average over the family of primitive characters modulo $q$.
\end{abstract}

\maketitle

\section{Introduction}\label{sect:intro}

Let $\chi$ be a primitive Dirichlet character modulo $q$, and let $L(s,\chi)$ be the associated Dirichlet $L$-function, with non-trivial zeros $\rho=\tfrac12+i\gamma$. The Generalised Riemann Hypothesis (GRH) for $L(s,\chi)$ asserts that $\gamma$ is real for all such zeros. In this paper we study two discrete second moments of $L(s,\chi)$, taken over its non-trivial zeros, that generalise to the Dirichlet setting two conjectures and theorems known for the Riemann zeta-function. We emphasise that our results are uniform in the conductor $q$.

\subsection{Background: the moments at $q=1$}
On the assumption of the Riemann Hypothesis (RH) and the simplicity of the zeros, Gonek made the following conjecture in his Mathematical Sciences Research Institute talk \cite{Gonek}.

\begin{conjecture}[Gonek]\label{conj:gonek}
Assume the Riemann Hypothesis and that the non-trivial zeros $\rho=\tfrac12+i\gamma$ of $\zeta(s)$ are simple. Then, as $T\to\infty$,
\[
\sum_{0<\gamma\le T}\frac{1}{|\zeta'(\rho)|^2}\sim\frac{3}{\pi^3}\,T.
\]
\end{conjecture}

Milinovich and Ng \cite{MilinovichNg} established a lower bound for this sum that is half the conjectured size.

\begin{theorem}[Milinovich--Ng]\label{thm:MNzeta}
Assume the Riemann Hypothesis and that the non-trivial zeros of $\zeta(s)$ are simple. Then, for every fixed $\varepsilon>0$ and all $T$ sufficiently large,
\[
\sum_{0<\gamma\le T}\frac{1}{|\zeta'(\rho)|^2}\ge\Bigl(\frac{3}{2\pi^3}-\varepsilon\Bigr)T.
\]
\end{theorem}

The companion moment was conjectured by Ng \cite{NgThesis,NgPNET} and appears as a problem of the Comparative Prime Number Theory symposium \cite{CPNT}.

\begin{conjecture}[Ng]\label{conj:ng}
Assume the Riemann Hypothesis and that the non-trivial zeros of $\zeta(s)$ are simple. Then, as $T\to\infty$,
\[
\sum_{0<\gamma\le T}\left|\frac{\zeta(2\rho)}{\zeta'(\rho)}\right|^2\sim\frac{T}{2\pi}.
\]
\end{conjecture}

The matching lower bound, again half the conjectured size, was proved by Sinha \cite{SinhaThesis}; an independent proof, differing in its treatment of the critical contour, was given by the author \cite{PCNg}.

\begin{theorem}[Sinha]\label{thm:Sinha}
Assume the Riemann Hypothesis and that the non-trivial zeros of $\zeta(s)$ are simple. Then, for every fixed $\varepsilon>0$ and all $T$ sufficiently large,
\[
\sum_{0<\gamma\le T}\left|\frac{\zeta(2\rho)}{\zeta'(\rho)}\right|^2\ge\Bigl(\frac{1}{4\pi}-\varepsilon\Bigr)T.
\]
\end{theorem}

\subsection{The Dirichlet moments, uniformly in the conductor}
The analogous Dirichlet moments are dictated by two Dirichlet-series identities. The M\"obius function is defined by
\[
\mu(n)=
\begin{cases}
1 & \text{if } n=1,\\
(-1)^{k} & \text{if } n=p_1\dotsm p_k \text{ for distinct primes } p_1,\dotsc,p_k,\\
0 & \text{if } n \text{ is not square-free},
\end{cases}
\]
and gives, for $\Re s>1$,
\begin{equation}\label{eq:mobius-id}
\frac{1}{L(s,\chi)}=\sum_{n\ge1}\frac{\mu(n)\chi(n)}{n^{s}},
\end{equation}
so that the residue of $1/L(s,\chi)$ at a simple zero $\rho$ is $1/L'(\rho,\chi)$. We refer to the corresponding sum 
$$\sum_{0<\gamma\le T}\frac{1}{|L'(\rho,\chi)|^2}$$
as the reciprocal moment. The Liouville function is defined by 
\begin{equation*}
\lambda(n) =
\begin{cases}
1 & \text{if } n = 1, \\
(-1)^{\Omega(n)} & \text{if } n > 1,
\end{cases}
\end{equation*}
where $\Omega(n)$ denotes the number of prime factors of $n$ counted with multiplicity, and gives, for $\Re s>1$,
\begin{equation}\label{eq:liouville-id}
\frac{L(2s,\chi^2)}{L(s,\chi)} = \sum_{n\ge1}\frac{\lambda(n)\chi(n)}{n^s},
\end{equation}
where $\chi^2$ denotes the character $n\mapsto\chi(n)^2$, so that the residue of $L(2s,\chi^2)/L(s,\chi)$ at a simple zero $\rho$ is $L(2\rho,\chi^2)/L'(\rho,\chi)$. We refer to the corresponding sum 
$$\sum_{0<\gamma\le T}\left|\frac{L(2\rho,\chi^2)}{L'(\rho,\chi)}\right|^2$$ 
as the ratio moment.

When $q$ is allowed to grow with $T$, the strength of the bounds we obtain is governed by the single quantity
\begin{equation}\label{eq:beta}
\beta=\beta(q,T)=\frac{\log T}{\log qT}=\Bigl(1+\frac{\log q}{\log T}\Bigr)^{-1}\in(0,1).
\end{equation}

\begin{theorem}\label{thm:MN}
Fix $A>0$ and $\varepsilon>0$. Assume GRH for $L(s,\chi)$ and that all of its non-trivial zeros are simple. Then, uniformly for all primitive characters $\chi$ modulo $q$ with $3\le q\le T^{A}$, and all $T$ sufficiently large in terms of $A$ and $\varepsilon$,
\[
\sum_{0<\gamma\le T}\frac{1}{|L'(\rho,\chi)|^2}
\;\ge\;(1-\varepsilon)\,\frac{A_q}{2\pi}\,\frac{\beta}{1+\beta}\,T
\;=\;(1-\varepsilon)\,\frac{A_q}{2\pi\bigl(2+\frac{\log q}{\log T}\bigr)}\,T ,
\]
with $\beta$ as in \eqref{eq:beta}, and where $A_q$ is the arithmetic factor
\begin{equation}\label{eq:Aq}
A_q=\frac{1}{\zeta(2)}\prod_{p\mid q}\frac{p}{p+1}=\frac{6}{\pi^2}\prod_{p\mid q}\frac{p}{p+1}.
\end{equation}
\end{theorem}

\begin{theorem}\label{thm:Ng}
Fix $0<\delta<1$ and $\varepsilon>0$. Assume GRH for $L(s,\chi)$ and that all of its non-trivial zeros are simple. Then, uniformly for all primitive characters $\chi$ modulo $q$ with $3\le q\le T^{1-\delta}$, and all $T$ sufficiently large in terms of $\delta$ and $\varepsilon$,
\[
\sum_{0<\gamma\le T}\left|\frac{L(2\rho,\chi^2)}{L'(\rho,\chi)}\right|^2
\;\ge\;(1-\varepsilon)\,\frac{\varphi(q)}{q}\,\frac{1}{2\pi}\,\frac{\beta}{1+\beta}\,T
\;=\;(1-\varepsilon)\,\frac{\varphi(q)}{q}\,\frac{1}{2\pi\bigl(2+\frac{\log q}{\log T}\bigr)}\,T,
\]
where $\varphi$ is the Euler totient function, given by
\[
\varphi(q)=\#\{1\le a\le q:(a,q)=1\}=q\prod_{p\mid q}\Bigl(1-\frac1p\Bigr).
\]
\end{theorem}

The ranges of the two theorems differ because we do not assume GRH for the second character $\chi^2$ appearing in the ratio moment: at one point in the proof of Theorem~\ref{thm:Ng} the factor $L(s,\chi^2)$ can only be bounded by convexity, and this is what restricts the range to $q\le T^{1-\delta}$. If one is willing to assume GRH for $L(s,\chi^2)$ as well, the restriction disappears and the range may be taken to be $q\le T^{A}$ as in Theorem~\ref{thm:MN}; we make this precise in Section~\ref{sect:propsNg}.

In what follows we write
\begin{equation}\label{eq:aq}
\mathfrak a(q)=
\begin{cases}
A_q & \text{for the reciprocal moment},\\[2pt]
\varphi(q)/q & \text{for the ratio moment},
\end{cases}
\end{equation}
for the arithmetic factor of the two problems. Two regimes deserve to be singled out: small conductors, where the bounds take their cleanest form, and power conductors, where the conductor begins to cost.

\begin{remark}
If $q=q(T)$ satisfies $\log q=o(\log T)$ (in particular if $q$ is fixed, and more generally if $q\le(\log T)^B$ for any fixed $B$, or $q=\exp\bigl((\log T)^{1-\delta}\bigr)$), then $\beta\to1$ and, uniformly in this range,
\[
\sum_{0<\gamma\le T}\frac{1}{|L'(\rho,\chi)|^2}\ge(1-\varepsilon)\,\frac{A_q}{4\pi}\,T .
\]
\end{remark}

\begin{remark}
If $q=q(T)$ satisfies $\log q=o(\log T)$, then, uniformly in this range,
\[
\sum_{0<\gamma\le T}\left|\frac{L(2\rho,\chi^2)}{L'(\rho,\chi)}\right|^2\ge(1-\varepsilon)\,\frac{\varphi(q)}{q}\,\frac{1}{4\pi}\,T .
\]
\end{remark}

At $q=1$ the constants $A_1/4\pi=3/2\pi^3$ and $\tfrac{1}{4\pi}$ recover Theorems~\ref{thm:MNzeta} and~\ref{thm:Sinha}; each of these two bounds is half of the expected main term (Conjectures~\ref{conj:dirMN} and~\ref{conj:dirNg} below).

\begin{remark}
If $q=T^{A}$ with $A>0$ fixed, then $\beta=1/(1+A)$, so that
\[
\frac{\beta}{1+\beta}=\frac{1/(1+A)}{(2+A)/(1+A)}=\frac{1}{2+A},
\]
and the proved constant is $\frac{\mathfrak a(q)}{2\pi}\,\frac{1}{2+A}$, with $\mathfrak a(q)$ as in \eqref{eq:aq}. That is, the factor $\tfrac12$ of the small-conductor regime degrades to $\tfrac{1}{2+A}$ of the expected main term: $\tfrac25$ of it at $A=\tfrac12$, $\tfrac13$ at $A=1$, $\tfrac14$ at $A=2$. (For the ratio moment, Theorem~\ref{thm:Ng} requires $A<1$.)
\end{remark}

\begin{remark}
The arithmetic factor itself decays only very slowly: by Mertens' theorem $\mathfrak a(q)\gg1/\log\log q$, so the bounds remain of the exact order $\mathfrak a(q)T$ of the expected main term throughout the stated ranges, losing at most a factor $\ll\log\log q$ against $T$. Beyond those ranges, the method continues to give a positive constant until the conditional size estimate $\exp\bigl(c\log qT/\log\log qT\bigr)=T^{o(1)}$ fails, that is, for $q$ up to $T^{o(\log\log T)}$; we do not pursue this outer range.
\end{remark}

\subsection{The expected size, and the source of the conductor deficit}\label{subsect:expected}
We do not import an external prediction for the size of these moments unlike the zeta function case. Rather, the method itself suggests one, and at the same time explains the factor $\beta/(1+\beta)$. As we show in Section~\ref{sect:opt}, the optimised Cauchy--Schwarz inequality produces the quantity
\begin{equation}\label{eq:master-intro}
\frac{\mathfrak a(q)}{2\pi}\,\frac{u}{1+u}\,T,
\qquad
u=\frac{\log\xi}{\log qT},
\end{equation}
where $\xi$ is the length of the mollifier. The mean value theorem for Dirichlet polynomials caps the length at $\xi\le T^{1-\varepsilon}$ independently of $q$ (a single $L$-function observed up to height $T$ resolves frequencies only up to $T$), so $u\le\beta$, and the supremum of \eqref{eq:master-intro} is the bound of Theorems~\ref{thm:MN} and~\ref{thm:Ng}. In the opposite direction, Farmer's long-mollifier heuristic \cite{Farmer} suggests that the value of \eqref{eq:master-intro} as $u\to\infty$, namely $\tfrac{\mathfrak a(q)}{2\pi}T$, is the true size of each moment. This leads us to the following two conjectures.

\begin{conjecture}\label{conj:dirMN}
Let $\chi$ be a primitive Dirichlet character modulo $q>1$. Assume GRH for $L(s,\chi)$ and that all of its non-trivial zeros are simple. Then, as $T\to\infty$,
\[
\sum_{0<\gamma\le T}\frac{1}{|L'(\rho,\chi)|^2}\sim\frac{A_q}{2\pi}\,T,
\]
and we expect this asymptotic to hold uniformly for $q\le T^{A}$, for any fixed $A>0$.
\end{conjecture}

\begin{conjecture}\label{conj:dirNg}
Let $\chi$ be a primitive Dirichlet character modulo $q>1$. Assume GRH for $L(s,\chi)$ and that all of its non-trivial zeros are simple. Then, as $T\to\infty$,
\[
\sum_{0<\gamma\le T}\left|\frac{L(2\rho,\chi^2)}{L'(\rho,\chi)}\right|^2\sim\frac{\varphi(q)}{q}\,\frac{T}{2\pi},
\]
and we expect this asymptotic to hold uniformly for $q\le T^{A}$, for any fixed $A>0$.
\end{conjecture}

Thus, conjecturally, our theorems capture exactly the proportion $\tfrac{\beta}{1+\beta}$ of the truth, and in the small-conductor regime exactly half of it: the deficiency is the ratio of the value of $u/(1+u)$ at the cap $u=\beta$ to its limiting value $1$, a structural feature of the diagonal Cauchy--Schwarz construction rather than a numerical accident. Conjectures~\ref{conj:dirMN} and~\ref{conj:dirNg} reduce to Conjectures~\ref{conj:gonek} and~\ref{conj:ng} at $q=1$, which lends them support; we discuss their provenance further in Section~\ref{sect:opt}. It would be of interest to see whether they can be derived from the recipe of Conrey--Farmer--Keating--Rubinstein--Snaith \cite{CFKRS} or from the $L$-functions Ratios Conjecture \cite{CFZ,ConreySnaith}; we do not pursue this here.

\subsection{Averages over the family}\label{subsect:families}
The constants in Conjectures~\ref{conj:dirMN} and~\ref{conj:dirNg} depend only on the modulus and not on the character, so conjectures for the averages over all primitive characters modulo $q$ follow by additivity. We write $\sum{^*}_{\chi\bmod q}$ for a sum over these characters and $\varphi^*(q)$ for their number, so that
\begin{equation}\label{eq:phistar}
\varphi^*(q)=\sum_{d\mid q}\mu(q/d)\varphi(d)=q\prod_{p\,\|\,q}\Bigl(1-\frac2p\Bigr)\prod_{p^2\mid q}\Bigl(1-\frac1p\Bigr)^{2},
\end{equation}
where $p\,\|\,q$ means that $p\mid q$ but $p^2\nmid q$; in particular $\varphi^*(q)=0$ precisely when $q\equiv2\pmod4$, in which case there are no primitive characters and the statements below are empty.

\begin{conjecture}\label{conj:famA}
Let $q\not\equiv2\pmod4$, and assume GRH for $L(s,\chi)$ and the simplicity of its zeros for every primitive character $\chi$ modulo $q$. Then, as $T\to\infty$,
\[
\frac{1}{\varphi^*(q)} \sideset{}{^*}\sum_{\chi\bmod q}\;\sum_{0<\gamma_\chi\le T}\frac{1}{|L'(\rho,\chi)|^2}\sim \frac{A_q}{2\pi}\,T,
\]
and we expect this asymptotic to hold uniformly for $q\le T^{A}$, for any fixed $A>0$.
\end{conjecture}

\begin{conjecture}\label{conj:famB}
Let $q\not\equiv2\pmod4$, and assume GRH for $L(s,\chi)$ and the simplicity of its zeros for every primitive character $\chi$ modulo $q$. Then, as $T\to\infty$,
\[
\frac{1}{\varphi^*(q)}  \sideset{}{^*}\sum_{\chi\bmod q}\;\sum_{0<\gamma_\chi\le T}\left|\frac{L(2\rho,\chi^2)}{L'(\rho,\chi)}\right|^2\sim \frac{\varphi(q)}{q}\,\frac{T}{2\pi},
\]
and we expect this asymptotic to hold uniformly for $q\le T^{A}$, for any fixed $A>0$.
\end{conjecture}

We explain the reasoning behind Conjectures~\ref{conj:famA} and~\ref{conj:famB}, and comment on the prospects for a uniform treatment of the family, in Section~\ref{sect:family}.

\subsection{Higher moments}\label{sec:higher}

The lower bounds of Theorems~\ref{thm:MN} and~\ref{thm:Ng} propagate, by a
single application of H\"older's inequality, to lower bounds for all higher
moments
\[
  \sum_{0<\gamma\le T}\frac{1}{|L'(\rho,\chi)|^{2k}}
  \qquad\text{and}\qquad
  \sum_{0<\gamma\le T}\left|\frac{L(2\rho,\chi^2)}{L'(\rho,\chi)}\right|^{2k},
\]
where $k \geq 1$, uniformly in the conductor. For $k>1$ these are of order $T/(\log qT)^{\,k-1}=o(T)$, far below the conjectured sizes. A slight advantage is that our lower bounds are explicit, but we emphasise that they are not good lower bounds in terms of what we expect the truth to be.

\begin{corollary}\label{cor:MN-higher}
Fix $A>0$, a real number $k\ge1$, and $\varepsilon>0$. Assume the Generalised
Riemann Hypothesis for $L(s,\chi)$ and that all of its non-trivial zeros are
simple. Then, uniformly for all primitive characters $\chi$ modulo $q$ with
$3\le q\le T^{A}$ and all $T$ sufficiently large in terms of $A$, $k$ and
$\varepsilon$,
\[
  \sum_{0<\gamma\le T}\frac{1}{|L'(\rho,\chi)|^{2k}}
  \;\ge\;(1-\varepsilon)\,\frac{A_q^{\,k}}{2\pi}
  \Bigl(\frac{\beta}{1+\beta}\Bigr)^{k}\frac{T}{(\log qT)^{\,k-1}}
  \;=\;(1-\varepsilon)\,\frac{A_q^{\,k}}{2\pi\bigl(2+\frac{\log q}{\log T}\bigr)^{k}}
  \,\frac{T}{(\log qT)^{\,k-1}},
\]
with $\beta$ as in \eqref{eq:beta} and $A_q$ as in \eqref{eq:Aq}. In particular,
if $\log q=o(\log T)$ then $\beta\to1$ and $\log qT\sim\log T$, so that
\[
  \sum_{0<\gamma\le T}\frac{1}{|L'(\rho,\chi)|^{2k}}
  \;\ge\;(1-\varepsilon)\,\frac{A_q^{\,k}}{2^{\,k+1}\pi}\,\frac{T}{(\log T)^{\,k-1}} .
\]
\end{corollary}

\begin{remark}
In the case of the Riemann zeta-function ($q=1$), the negative discrete moments $\sum_{0<\gamma\le T}|\zeta'(\rho)|^{-2k}$ have been studied directly, and far more precisely than the H\"older bound of Corollary~\ref{cor:MN-higher} permits. Heap, Li and Zhao \cite{HLZ} proved lower bounds for all rational $k\ge0$, of the order predicted by the conjecture of Gonek and Hejhal, and Gao and Zhao \cite{GaoZhao} removed the rationality restriction, obtaining the same for all real $k\ge0$, with bounds expected to be sharp for $k<\tfrac32$; in the other direction Bui, Florea and Milinovich \cite{BFM} obtained upper bounds, almost optimal for $k\le\tfrac12$ and averaged over a subfamily of the zeros of asymptotically full density. The content of Corollary~\ref{cor:MN-higher} is thus its uniformity in the conductor rather than its strength.
\end{remark}

\begin{corollary}\label{cor:Ng-higher}
Fix $0<\delta<1$, a real number $k\ge1$, and $\varepsilon>0$. Assume the
Generalised Riemann Hypothesis for $L(s,\chi)$ and that all of its non-trivial
zeros are simple. Then, uniformly for all primitive characters $\chi$ modulo $q$
with $3\le q\le T^{1-\delta}$ and all $T$ sufficiently large in terms of $\delta$,
$k$ and $\varepsilon$,
\begin{align*}
      \sum_{0<\gamma\le T}\left|\frac{L(2\rho,\chi^2)}{L'(\rho,\chi)}\right|^{2k}
  \;&\ge\;(1-\varepsilon)\,\frac{\varphi(q)^{k}}{q^{k}}\,\frac{1}{2\pi}
  \Bigl(\frac{\beta}{1+\beta}\Bigr)^{k}\frac{T}{(\log qT)^{\,k-1}} \\
  &=\;(1-\varepsilon)\,\frac{\varphi(q)^{k}}{q^{k}}\,
  \frac{1}{2\pi\bigl(2+\frac{\log q}{\log T}\bigr)^{k}}\,\frac{T}{(\log qT)^{\,k-1}},
\end{align*}
with $\beta$ as in \eqref{eq:beta} and $\varphi$ the Euler totient function. In
particular, if $\log q=o(\log T)$ then $\beta\to1$ and $\log qT\sim\log T$, giving
\[
  \sum_{0<\gamma\le T}\left|\frac{L(2\rho,\chi^2)}{L'(\rho,\chi)}\right|^{2k}
  \;\ge\;(1-\varepsilon)\,\frac{\varphi(q)^{k}}{2^{\,k+1}\pi\,q^{k}}\,
  \frac{T}{(\log T)^{\,k-1}} .
\]
\end{corollary}

\begin{remark}
A refinement of these methods, developed in a follow-up paper, narrows the deficit towards the conjectured size in both the reciprocal and ratio moments.
\end{remark}

\begin{remark}
The corresponding family bounds follow in an analogous way. Since the single-character bounds above depend on $\chi$ only through the modulus $q$, averaging over the $\varphi^*(q)$ primitive characters modulo $q$ shows that the family averages
\[
\frac{1}{\varphi^*(q)}\sideset{}{^*}\sum_{\chi\bmod q}\;\sum_{0<\gamma_\chi\le T}\frac{1}{|L'(\rho,\chi)|^{2k}}
\qquad\text{and}\qquad
\frac{1}{\varphi^*(q)}\sideset{}{^*}\sum_{\chi\bmod q}\;\sum_{0<\gamma_\chi\le T}\left|\frac{L(2\rho,\chi^2)}{L'(\rho,\chi)}\right|^{2k}
\]
obey the same lower bounds as the single-character moments of Corollaries~\ref{cor:MN-higher} and~\ref{cor:Ng-higher}, in the ranges $3\le q\le T^{A}$ and $3\le q\le T^{1-\delta}$ respectively.
\end{remark}

The paper is organised as follows. Section~\ref{sect:setup} sets up the two mollifiers and reduces Theorems~\ref{thm:MN} and~\ref{thm:Ng} to four propositions, Propositions~\ref{prop:M1MN} to~\ref{prop:M2Ng}, which evaluate the mollified first and second moments of each problem with conductor-uniform error terms. Section~\ref{sect:lemmas} collects the preliminary lemmas, each stated uniformly in $q$. Sections~\ref{sect:propsMN} and~\ref{sect:propsNg} prove the propositions for the reciprocal and the ratio moments respectively, carrying the conductor through every estimate. Section~\ref{sect:opt} combines the propositions, optimises over the mollifier, proves the theorems and corollaries, and explains how Conjectures~\ref{conj:dirMN} and~\ref{conj:dirNg} arise from the same computation. Section~\ref{sect:family} discusses the averages over the family of primitive characters modulo $q$ and the reasoning behind Conjectures~\ref{conj:famA} and~\ref{conj:famB}. Finally, Section~\ref{sec:higher-proofs} discusses the proofs of the higher moments stated above, namely Corollaries~\ref{cor:MN-higher} and~\ref{cor:Ng-higher} together with their family forms.

\section{Setup of the theorems}\label{sect:setup}

We follow the strategy of Milinovich and Ng \cite{MilinovichNg}, who in turn based their argument on the work of Rudnick and Soundararajan \cite{RudnickSound}. Throughout the remainder of the paper $\chi$ is a primitive character modulo $q\ge3$, $\bar\chi$ is its complex conjugate, and we assume GRH for $L(s,\chi)$ together with the simplicity of its zeros. The parameters $A>0$ (for Theorem~\ref{thm:MN}), $\delta\in(0,1)$ (for Theorem~\ref{thm:Ng}), $\vartheta\in(0,1)$ and the polynomial $P$ below are fixed. All implied constants are absolute or depend only on these fixed parameters, never on $q$ or $\chi$, and all estimates are uniform over $3\le q\le T^{A}$ (respectively $3\le q\le T^{1-\delta}$ for the ratio moment, where indicated).

Let $\xi=T^{\vartheta}$, and let $P(x)$ be a polynomial satisfying $P(0)=0$ and $P(1)=1$. For brevity we write
\[
P_n=P\!\left(\frac{\log(\xi/n)}{\log\xi}\right).
\]
The two theorems are attacked with two mollifiers, one for each. For Theorem~\ref{thm:MN} we use the M\"obius-twisted Dirichlet polynomial
\[
\mathcal{L}_{\psi}(s)=\sum_{n\le\xi}\frac{\mu(n)\psi(n)}{n^{s}}\,P_n
\qquad(\psi\in\{\chi,\bar\chi\}),
\]
which is a smoothed truncation of the Dirichlet series \eqref{eq:mobius-id}. For Theorem~\ref{thm:Ng} we use the Liouville-twisted polynomial
\[
\widetilde{\mathcal{L}}_{\psi}(s)=\sum_{n\le\xi}\frac{\lambda(n)\psi(n)}{n^{s}}\,P_n
\qquad(\psi\in\{\chi,\bar\chi\}),
\]
which by \eqref{eq:liouville-id} is a smoothed truncation of $L(2s,\chi^2)/L(s,\chi)$. In each case the residue of the underlying Dirichlet series at a simple zero $\rho$ of $L(s,\chi)$ is precisely the summand of the corresponding moment: $1/L'(\rho,\chi)$ for the reciprocal moment and $L(2\rho,\chi^2)/L'(\rho,\chi)$ for the ratio moment. This is what motivates the choice of mollifier in each case.

Since $\overline{L(s,\chi)}=L(\bar s,\bar\chi)$, and since under GRH every non-trivial zero satisfies $\bar\rho=1-\rho$, both mollifiers enjoy the symmetry
\begin{equation}\label{eq:symm}
\overline{\mathcal{L}_{\chi}(\rho)}=\mathcal{L}_{\bar\chi}(1-\rho),
\qquad
\overline{\widetilde{\mathcal{L}}_{\chi}(\rho)}=\widetilde{\mathcal{L}}_{\bar\chi}(1-\rho).
\end{equation}
Writing the summand of the reciprocal moment as $1/L'(\rho,\chi)$ and applying the Cauchy--Schwarz inequality to $\sum_{0<\gamma\le T}\frac{1}{L'(\rho,\chi)}\overline{\mathcal{L}_\chi(\rho)}$, the symmetry \eqref{eq:symm} gives
\begin{equation}\label{eq:CS-MN}
\sum_{0<\gamma\le T}\frac{1}{|L'(\rho,\chi)|^2}\ge\frac{|M_1|^2}{M_2},
\end{equation}
where
\[
M_1=\sum_{0<\gamma\le T}\frac{\mathcal{L}_{\bar\chi}(1-\rho)}{L'(\rho,\chi)}
\qquad\text{and}\qquad
M_2=\sum_{0<\gamma\le T}|\mathcal{L}_{\chi}(\rho)|^2 .
\]
Likewise, applying the Cauchy--Schwarz inequality to $\sum_{0<\gamma\le T}\frac{L(2\rho,\chi^2)}{L'(\rho,\chi)}\overline{\widetilde{\mathcal{L}}_\chi(\rho)}$ and using the second relation in \eqref{eq:symm}, we obtain for the ratio moment
\begin{equation}\label{eq:CS-Ng}
\sum_{0<\gamma\le T}\left|\frac{L(2\rho,\chi^2)}{L'(\rho,\chi)}\right|^2\ge\frac{|\widetilde M_1|^2}{\widetilde M_2},
\end{equation}
where
\[
\widetilde M_1=\sum_{0<\gamma\le T}\frac{L(2\rho,\chi^2)}{L'(\rho,\chi)}\widetilde{\mathcal{L}}_{\bar\chi}(1-\rho)
\qquad\text{and}\qquad
\widetilde M_2=\sum_{0<\gamma\le T}|\widetilde{\mathcal{L}}_{\chi}(\rho)|^2 .
\]
Our task is therefore to obtain an asymptotic for each of $M_1,M_2,\widetilde M_1,\widetilde M_2$, uniformly in $q$, and then to choose $P$ so as to make the resulting lower bounds as large as possible.

One technical convention is needed before we begin. The contour rectangles of Sections~\ref{sect:propsMN} and~\ref{sect:propsNg} have their bottom edge at a height $t_0=t_0(\chi)\in[1,2]$, chosen in Lemma~\ref{lem:heights} so as to avoid the low-lying zeros uniformly in $q$, and the residue theorem therefore produces sums over the zeros with $t_0<\gamma\le T$ rather than $0<\gamma\le T$. This costs nothing, for two reasons. First, the Cauchy--Schwarz inequalities \eqref{eq:CS-MN} and \eqref{eq:CS-Ng} are valid over any set of zeros, so we are free to define $M_1,M_2,\widetilde M_1,\widetilde M_2$ by sums over $t_0<\gamma\le T$. Secondly, the summands of the two moments are non-negative, so a lower bound for the sum restricted to $t_0<\gamma\le T$ is also a lower bound for the full sum over $0<\gamma\le T$. We adopt this convention without further comment.

\begin{remark}
The two hypotheses play distinct roles. The simplicity of the zeros guarantees that the moments are well-defined ($L'(\rho,\chi)\ne0$) and that the poles of $1/L(s,\chi)$ and of $L(2s,\chi^2)/L(s,\chi)$ at the zeros are simple, with the stated residues, in the contour integrals of Sections~\ref{sect:propsMN} and~\ref{sect:propsNg}. GRH for $L(s,\chi)$ places all the zeros on the critical line, which is what yields the symmetry \eqref{eq:symm}; it is also used through the size estimates of Sections~\ref{sect:propsMN} and~\ref{sect:propsNg}. No hypothesis on $L(s,\chi^2)$ is required.
\end{remark}

The arithmetic of the two problems enters only through one coprimality sum: on the diagonal $m=n$ the character contribution is $\chi(n)\bar\chi(n)=|\chi(n)|^2=\mathbf 1_{(n,q)=1}$, and the relevant densities, with conductor-uniform secondary terms, are recorded in Lemma~\ref{lem:densities} below. The conductor itself enters the main terms in exactly one place: the zero-counting density and the additive functional equation carry the analytic-conductor logarithm $\log qT$ rather than $\log T$, and this survives precisely in the second-moment term $K$ of Sections~\ref{sect:propsMN} and~\ref{sect:propsNg}. We can now state the four asymptotic estimates. The first pair concerns the reciprocal moment.

\begin{proposition}\label{prop:M1MN}
Assume GRH for $L(s,\chi)$ and that its zeros are simple, and let $0<\vartheta<1$ and $A>0$ be fixed. Then there is a sequence $\mathcal T=\{\tau_m\}$ of admissible heights, with bounded gaps, such that, uniformly for $3\le q\le T^{A}$ and $T\in\mathcal T$,
\[
M_1=\frac{T}{2\pi}\,A_q\log T\;\vartheta\!\int_0^1 P(x)\,dx+O\bigl(T\log\log 3q\bigr).
\]
\end{proposition}

\begin{proposition}\label{prop:M2MN}
Assume GRH for $L(s,\chi)$ and that its zeros are simple, and let $0<\vartheta<1$ and $A>0$ be fixed. Then, uniformly for $3\le q\le T^{A}$ and $T$ in the sequence of Lemma~\ref{lem:heights},
\[
M_2=\frac{T}{2\pi}\,A_q\log T\left[\log qT\;\vartheta\!\int_0^1 P(x)^2\,dx+\log T\;\vartheta^2\Bigl(\int_0^1 P(x)\,dx\Bigr)^2\right]+O\bigl(T\log T\log\log 3qT\bigr).
\]
\end{proposition}

The second pair concerns the ratio moment, and differs in the arithmetic factor and in the admissible range of $q$.

\begin{proposition}\label{prop:M1Ng}
Assume GRH for $L(s,\chi)$ and that its zeros are simple, and let $0<\vartheta<1$ and $0<\delta<1$ be fixed. Then there is a sequence $\mathcal T$ as above such that, uniformly for $3\le q\le T^{1-\delta}$ and $T\in\mathcal T$,
\[
\widetilde M_1=\frac{\varphi(q)}{q}\,\frac{T}{2\pi}\log T\;\vartheta\!\int_0^1 P(x)\,dx+O\bigl(T\log\log 3q\bigr).
\]
\end{proposition}

\begin{proposition}\label{prop:M2Ng}
Assume GRH for $L(s,\chi)$ and that its zeros are simple, and let $0<\vartheta<1$ and $A>0$ be fixed. Then, uniformly for $3\le q\le T^{A}$ and $T$ in the sequence of Lemma~\ref{lem:heights},
\[
\widetilde M_2=\frac{\varphi(q)}{q}\,\frac{T}{2\pi}\log T\left[\log qT\;\vartheta\!\int_0^1 P(x)^2\,dx+\log T\;\vartheta^2\Bigl(\int_0^1 P(x)\,dx\Bigr)^2\right]+O\bigl(T\log T\log\log 3qT\bigr).
\]
\end{proposition}

\begin{remark}
For fixed $q$ one has $\log qT=\log T+O(1)$ and $\log\log 3qT=O(1)$, and the propositions reduce to their classical shapes, with second moments $\tfrac{T(\log T)^2}{2\pi}\mathfrak a(q)\bigl[\vartheta\int P^2+\vartheta^2(\int P)^2\bigr]$. Since $\mathfrak a(q)\gg1/\log\log q$ by Mertens' theorem, the error terms above are smaller than the corresponding main terms by the relative factor $O\bigl((\log\log qT)^2/\log T\bigr)$, uniformly in the stated ranges; this is what makes the multiplicative form of Theorems~\ref{thm:MN} and~\ref{thm:Ng} possible.
\end{remark}

In Section~\ref{sect:opt} we substitute these estimates into \eqref{eq:CS-MN} and \eqref{eq:CS-Ng}, optimise over $P$, and prove the theorems and the corollaries.

\section{Preliminary lemmas}\label{sect:lemmas}

We collect here the analytic facts used in the proofs. In contrast to the fixed-conductor setting, every lemma below is stated with explicit uniformity in $q$; implied constants are absolute unless a dependence is displayed.

The central tool is the mean value theorem for Dirichlet polynomials of Montgomery and Vaughan \cite{MV}, in the bilinear form due to Tsang \cite{Tsang}.

\begin{lemma}\label{lem:MV}
Let $\{a_n\}_{n\ge1}$ and $\{b_n\}_{n\ge1}$ be sequences of complex numbers. Then for any real $T>0$,
\[
\int_0^T\Bigl(\sum_{n\ge1}a_n n^{-it}\Bigr)\Bigl(\sum_{n\ge1}b_n n^{it}\Bigr)\,dt
=T\sum_{n\ge1}a_n b_n+O\!\left(\Bigl(\sum_{n\ge1}n|a_n|^2\Bigr)^{1/2}\Bigl(\sum_{n\ge1}n|b_n|^2\Bigr)^{1/2}\right),
\]
with an absolute implied constant.
\end{lemma}

We remark at once that Lemma~\ref{lem:MV} is the source of the conductor-independent cap on the mollifier: its error term is of size $\xi$ times logarithms whatever the value of $q$ (the coefficients see $\chi$ only through $|\chi|\le1$), and must be dominated by main terms of size $T$ times logarithms; hence $\xi\le T^{1-\varepsilon}$, with no power of $q$ available on either side. Throughout we apply Lemma~\ref{lem:MV} with the $t$-integration over $[1,T]$ in place of $[0,T]$; the omitted $\int_0^1$ contributes $O(\xi\log T)=o(T)$ and is absorbed into the error terms.

Next we record the functional equation
\begin{equation}\label{eq:FE}
L(s,\chi)=X(s,\chi)L(1-s,\bar\chi)
\end{equation}
and the consequences of Stirling's formula for the factor $X$.

\begin{lemma}\label{lem:Xfactor}
For $s=\sigma+it$ with $\sigma$ in any fixed bounded interval and $|t|\ge1$, uniformly in $q$,
\[
|X(s,\chi)|\asymp\Bigl(\frac{q|t|}{2\pi}\Bigr)^{1/2-\sigma},
\qquad
\arg X(\sigma+it,\chi)=-t\log\frac{q|t|}{2\pi}+t+O(1),
\]
and the latter asymptotic may be differentiated termwise, giving $\frac{d}{dt}\arg X=-\log\frac{q|t|}{2\pi}+O(1/t)$ and $\frac{d^2}{dt^2}\arg X=-1/t+O(1/t^2)$; the conductor enters the phase only through its $t$-linear part. Moreover the logarithmic derivative satisfies, uniformly for $1\le|t|\le T$ and in $q$,
\[
\frac{L'}{L}(1-s,\bar\chi)=-\frac{L'}{L}(s,\chi)-\log\frac{q|t|}{2\pi}+O\!\left(\frac1{|t|}\right).
\]
\end{lemma}

Under GRH we have the following size and reciprocal-size estimates, with the conductor inside the exponent; the second part is the Dirichlet analogue of a classical theorem of Titchmarsh \cite[Theorem~14.16]{Titchmarsh}; in the conductor aspect these bounds are standard (see, e.g., \cite[Ch.~5]{IK}).

\begin{lemma}\label{lem:GRHsize}
Assume GRH for $L(s,\chi)$. There is an absolute constant $c_0>0$ such that, uniformly for $\sigma\ge\tfrac12$, $|t|\ge1$ and all $q$,
\[
|L(\sigma+it,\chi)|\ll\exp\!\left(\frac{c_0\log(q|t|)}{\log\log(q|t|)}\right)=(q|t|)^{o(1)} ,
\]
and the same bound holds for $|L(\sigma+it,\chi)|^{-1}$ at points bounded away from the zeros. In particular, for $q\le T^{A}$ and $|t|\le T$ these are $T^{o_A(1)}$, and there is a fixed $A_0>0$ and a sequence $\mathcal T=\{\tau_m\}_{m}$ with $m<\tau_m\le m+1$ on which
\[
|L(\sigma+i\tau_m,\chi)|^{-1}\ll\exp\!\left(\frac{A_0\log(q\tau_m)}{\log\log(q\tau_m)}\right)
\]
uniformly for $-\tfrac12\le\sigma\le 2$. For $\sigma<\tfrac12$ this follows from the case $\sigma\ge\tfrac12$ by the functional equation, since $|X(s,\chi)|^{-1}\ll1$ there and $\Re(1-s)\ge\tfrac12$.
\end{lemma}

For the horizontal segments of the contours we use the standard device of choosing the height to avoid the zeros (cf.\ Davenport \cite[p.~108]{Davenport}), with the gaps now measured against the conductor density.

\begin{lemma}\label{lem:heights}
There is a sequence of heights $T$, of bounded gaps, with $|\gamma-T|\gg1/\log(qT)$ for all ordinates $\gamma$ of $L(s,\chi)$, and
\[
\frac{L'}{L}(\sigma+iT,\chi)\ll(\log qT)^2
\]
uniformly for $1-c\le\sigma\le c$, where $c=1+1/\log T$. There is likewise a bottom height $t_0=t_0(\chi)\in[1,2]$ with $|\gamma-t_0|\gg1/\log 3q$ for all ordinates, on which $\frac{L'}{L}(\sigma+it_0,\chi)\ll(\log 3q)^2$ and, under GRH, $|L(\sigma+it_0,\chi)|^{-1}\ll q^{o(1)}$, uniformly for $1-c\le\sigma\le c$.
\end{lemma}

Such heights exist by the pigeonhole principle: the number of zeros in any interval of unit length is $O(\log qT)$ (respectively $O(\log 3q)$ near height $1$), so some sub-interval of length $\gg1/\log qT$ contains no ordinate. The sequence $\mathcal T$ of Lemma~\ref{lem:GRHsize} may be taken to satisfy the conditions of Lemma~\ref{lem:heights} as well.

The arithmetic densities, with conductor-uniform secondary terms, are as follows; they absorb what would otherwise be hidden $q$-dependence in the partial-summation steps, and they are the source of the arithmetic factor $A_q$.

\begin{lemma}\label{lem:densities}
Uniformly for $x\ge3$ and all $q\ge3$,
\[
\sum_{\substack{n\le x\\(n,q)=1}}\frac1n=\frac{\varphi(q)}{q}\log x+O\Bigl(\log\log 3q+\frac{2^{\omega(q)}}{x}\Bigr)
\]
and
\[
\sum_{\substack{n\le x\\(n,q)=1}}\frac{\mu(n)^2}{n}=A_q\log x+O\Bigl(\log\log 3q+\frac{2^{\omega(q)}}{\sqrt x}\Bigr),
\]
where $\omega(q)$ is the number of distinct prime factors of $q$. Moreover
\[
\sum_{p\mid q}\frac{\log p}{p}\ll\log\log 3q .
\]
\end{lemma}

\begin{proof}
For the prime sum, split at $y=2\log 3q$: Mertens' theorem gives $\sum_{p\le y}\frac{\log p}{p}\ll\log y\ll\log\log 3q$, while at most $\log q/\log y$ primes dividing $q$ exceed $y$, each contributing at most $\frac{\log y}{y}$, for a total $O(1)$.

For the first display, detect the coprimality condition by M\"obius inversion over the square-free divisors $d$ of $q$:
\[
\sum_{\substack{n\le x\\(n,q)=1}}\frac1n=\sum_{d\mid q}\frac{\mu(d)}{d}\sum_{m\le x/d}\frac1m
=\sum_{d\mid q}\frac{\mu(d)}{d}\Bigl(\log\frac xd+O(1)+O\Bigl(\frac dx\Bigr)\Bigr).
\]
The main term is $\bigl(\sum_{d\mid q}\mu(d)/d\bigr)\log x=\frac{\varphi(q)}{q}\log x$; the tail is $O(2^{\omega(q)}/x)$ (divisors $d>x$ give an empty inner sum, absorbed here since $\log(d/x)<d/x$); and the secondary term is controlled by the identity
\[
\sum_{d\mid q}\frac{\mu(d)\log d}{d}=-\frac{\varphi(q)}{q}\sum_{p\mid q}\frac{\log p}{p-1},
\]
which is $\ll\log\log 3q$ in absolute value, obtained by logarithmic differentiation of $\prod_{p\mid q}(1-p^{-s})=\sum_{d\mid q}\mu(d)d^{-s}$ at $s=1$, together with the prime-sum bound above.

For the second display, write $\mu(n)^2=\sum_{d^2\mid n}\mu(d)$, so that
\[
\sum_{\substack{n\le x\\(n,q)=1}}\frac{\mu(n)^2}{n}=\sum_{\substack{d\le\sqrt x\\(d,q)=1}}\frac{\mu(d)}{d^2}\sum_{\substack{m\le x/d^2\\(m,q)=1}}\frac1m .
\]
Applying the first display to the inner sum and summing the result over $d$, the constant in front of $\log x$ is
\[
\frac{\varphi(q)}{q}\sum_{(d,q)=1}\frac{\mu(d)}{d^2}
=\prod_{p\mid q}\Bigl(1-\frac1p\Bigr)\,\frac{1}{\zeta(2)}\prod_{p\mid q}\Bigl(1-\frac1{p^2}\Bigr)^{-1}
=\frac{1}{\zeta(2)}\prod_{p\mid q}\frac{p}{p+1}=A_q,
\]
which is \eqref{eq:Aq}, and the secondary terms are again $O(\log\log 3q+2^{\omega(q)}/\sqrt x)$.
\end{proof}

In our applications $x=\xi\ge T^{\vartheta}$ while $2^{\omega(q)}\le q^{o(1)}\le T^{o(1)}$, so the tails in Lemma~\ref{lem:densities} are negligible and we will quote the lemma with error $O(\log\log 3q)$ only; in the partial-summation steps the lemma is integrated against a weight of derivative $\ll(x\log\xi)^{-1}$, so the tail $2^{\omega(q)}/x$ contributes $\ll 2^{\omega(q)}/\log\xi=T^{o(1)}/\log T$, which is likewise negligible.

Finally, the oscillatory integrals in the proof of Proposition~\ref{prop:M1Ng} are estimated by van der Corput's second-derivative test \cite[Lemma~4.5]{Titchmarsh}.

\begin{lemma}\label{lem:vdc}
Let $F$ be real-valued and twice continuously differentiable on $[a,b]$, with $F''(t)$ of constant sign and $|F''(t)|\ge\Lambda>0$ throughout. Let $g$ be real-valued of bounded variation on $[a,b]$. Then
\[
\left|\int_a^b g(t)\,e^{iF(t)}\,dt\right|\ll\Lambda^{-1/2}\Bigl(\sup_{[a,b]}|g|+\operatorname{Var}_{[a,b]}g\Bigr),
\]
with an absolute implied constant.
\end{lemma}

Throughout the proofs we set $c=1+1/\log T$, so that
\[
\zeta(c)\asymp\log T,\qquad \zeta(2c-1)\asymp\log T,\qquad \zeta(2c)\asymp1,\qquad \xi^{c}\asymp\xi;
\]
these are genuinely $T$-only logarithms, a fact we will use repeatedly.

\section{Proof of Propositions \ref{prop:M1MN} and \ref{prop:M2MN}}\label{sect:propsMN}

Throughout this section the mollifier is $\mathcal{L}_\psi(s)=\sum_{n\le\xi}\mu(n)\psi(n)P_n n^{-s}$, and $3\le q\le T^{A}$.

\subsection{The first moment $M_1$}
Let $T\in\mathcal T$, with $\mathcal T$ the sequence of Lemma~\ref{lem:GRHsize}, and let $t_0\in[1,2]$ be the bottom height of Lemma~\ref{lem:heights}. Since the only non-trivial pole of $1/L(s,\chi)$ inside the rectangle $\mathcal R$ with vertices $c+it_0$, $c+iT$, $1-c+iT$, $1-c+it_0$ is a simple pole at each zero $\rho$, of residue $1/L'(\rho,\chi)$, Cauchy's residue theorem gives
\[
\begin{aligned}
M_1&=\frac{1}{2\pi i}\left(\int_{c+it_0}^{c+iT}+\int_{c+iT}^{1-c+iT}+\int_{1-c+iT}^{1-c+it_0}+\int_{1-c+it_0}^{c+it_0}\right)\frac{1}{L(s,\chi)}\,\mathcal{L}_{\bar\chi}(1-s)\,ds\\
&=I_1+I_2+I_3+I_4,
\end{aligned}
\]
say. The box captures exactly the zeros with $t_0<\gamma<T$, which is the normalisation fixed in Section~\ref{sect:setup}. (On the right and left edges we extend the $t$-integration from $t_0$ down to $1$ silently: the integrands there are $O\bigl(\xi(qT)^{o(1)}\bigr)$ pointwise on the affected segment of bounded length, a negligible change.) We show that $I_1$ supplies the main term and that $I_2,I_3,I_4$ are error terms, uniformly in $q$.

We begin with the right edge $I_1$. On the line $\Re s=c$ both factors have absolutely convergent Dirichlet series, $1/L(s,\chi)=\sum_{m}\mu(m)\chi(m)m^{-s}$ and $\mathcal{L}_{\bar\chi}(1-s)=\sum_{n\le\xi}\mu(n)\bar\chi(n)P_n\,n^{s-1}$. Writing $s=c+it$ and $ds=i\,dt$, we have
\[
I_1=\frac{1}{2\pi}\int_1^T\Bigl(\sum_{m}\frac{\mu(m)\chi(m)}{m^{c+it}}\Bigr)\Bigl(\sum_{n\le\xi}\frac{\mu(n)\bar\chi(n)P_n}{n^{1-c-it}}\Bigr)\,dt.
\]
We apply Lemma~\ref{lem:MV} with $a_m=\mu(m)\chi(m)m^{-c}$ and $b_n=\mu(n)\bar\chi(n)P_n n^{c-1}$. The main term is
\[
\frac{T}{2\pi}\sum_{n\le\xi}\mu(n)^2|\chi(n)|^2\frac{P_n}{n}=\frac{T}{2\pi}\sum_{\substack{n\le\xi\\(n,q)=1}}\frac{\mu(n)^2P_n}{n},
\]
since $\mu(n)\chi(n)\,\mu(n)\bar\chi(n)=\mu(n)^2|\chi(n)|^2$. For the error term,
\[
\sum_m m|a_m|^2=\sum_m\mu(m)^2|\chi(m)|^2m^{1-2c}\le\zeta(2c-1)\ll\log T,
\]
while
\[
\sum_{n\le\xi}n|b_n|^2=\sum_{n\le\xi}\mu(n)^2|\chi(n)|^2P_n^2\,n^{2c-1}\ll\xi^{2c}\ll\xi^2 ,
\]
both with absolute constants (the character enters only through $|\chi|\le1$), so the error is $O\bigl(\xi(\log T)^{1/2}\bigr)$. To evaluate the main sum we use the second density of Lemma~\ref{lem:densities} together with partial summation. Writing $u=\log n$ and noting that $P_n=P\bigl(1-u/\log\xi\bigr)$, we obtain
\[
\sum_{\substack{n\le\xi\\(n,q)=1}}\frac{\mu(n)^2P_n}{n}=A_q\log\xi\int_0^1 P(x)\,dx+O(\log\log 3q).
\]
Combining these, and recalling $\log\xi=\vartheta\log T$,
\begin{align}
I_1&=\frac{T}{2\pi}A_q(\log\xi)\int_0^1 P(x)\,dx+O\bigl(\xi(\log T)^{1/2}+T\log\log 3q\bigr)\notag\\
&=\frac{T}{2\pi}\,A_q\log T\;\vartheta\!\int_0^1 P(x)\,dx+O\bigl(T\log\log 3q\bigr),\label{eq:I1-MN}
\end{align}
since $\xi=T^{\vartheta}$ with $\vartheta<1$.

Next we bound the horizontal edges $I_2$ and $I_4$. On $I_2$ the path is $s=\sigma+iT$, $1-c\le\sigma\le c$, of length $O(1)$. By Lemma~\ref{lem:GRHsize} we have, for $T\in\mathcal T$ and $q\le T^{A}$,
\[
|L(\sigma+iT,\chi)|^{-1}\ll\exp\!\Bigl(\frac{A_0\log(qT)}{\log\log(qT)}\Bigr)=(qT)^{o(1)}=T^{o_A(1)},
\]
while the trivial bound gives $\mathcal{L}_{\bar\chi}(1-\sigma-iT)\ll\xi^{\max(\sigma,0)}\ll\xi$. Hence $I_2\ll\xi\,T^{o(1)}=o(T)$, uniformly in $q$. On the bottom edge, at the admissible height $t_0$ of Lemma~\ref{lem:heights}, we have $|L(\sigma+it_0,\chi)|^{-1}\ll q^{o(1)}$, so
\[
I_4\ll\xi\,q^{o(1)}=o(T).
\]

It remains to bound the left edge $I_3$. On $\Re s=1-c$ the functional equation \eqref{eq:FE} and Lemma~\ref{lem:Xfactor} give, since $\Re(1-s)=c>1$,
\[
\bigl|L(1-c+it,\chi)\bigr|^{-1}=\bigl|X(1-c+it,\chi)\bigr|^{-1}\bigl|L(c-it,\bar\chi)\bigr|^{-1}\ll(q|t|)^{-1/2}\log T,
\]
which decays in $t$, and here the conductor only helps. The logarithms are genuinely $\log T$, not $\log qT$: they arise from $|L(c-it,\bar\chi)|^{-1}\le\sum_n\mu(n)^2n^{-c}\le\zeta(c)$, which has no conductor in it. With the trivial bound $\mathcal{L}_{\bar\chi}(c-it)\ll\zeta(c)\ll\log T$,
\begin{equation}\label{eq:I3-MN}
I_3\ll\int_1^T(q|t|)^{-1/2}(\log T)^2\,dt\ll q^{-1/2}\,T^{1/2}(\log T)^2=o(T).
\end{equation}
Combining \eqref{eq:I1-MN}, the bounds for $I_2,I_4$, and \eqref{eq:I3-MN} proves Proposition~\ref{prop:M1MN}.

\subsection{The second moment $M_2$}
By the symmetry \eqref{eq:symm}, $M_2=\sum_{t_0<\gamma\le T}\mathcal{L}_\chi(\rho)\mathcal{L}_{\bar\chi}(1-\rho)$, and the residue of $\frac{L'}{L}(s,\chi)$ at a simple zero is $1$, so
\[
\begin{aligned}
M_2&=\frac{1}{2\pi i}\left(\int_{c+it_0}^{c+iT}+\int_{c+iT}^{1-c+iT}+\int_{1-c+iT}^{1-c+it_0}+\int_{1-c+it_0}^{c+it_0}\right)\frac{L'}{L}(s,\chi)\,\mathcal{L}_\chi(s)\,\mathcal{L}_{\bar\chi}(1-s)\,ds\\
&=J_1+J_2+J_3+J_4,
\end{aligned}
\]
say. We take $T$ in the sequence of Lemma~\ref{lem:heights}. On the horizontal edges the trivial bound $\mathcal{L}_\psi(\sigma+it)\ll\xi^{1-\sigma}$ together with $\frac{L'}{L}(\sigma+iT,\chi)\ll(\log qT)^2$ gives $J_2\ll\xi(\log qT)^2$ and, at the bottom height $t_0$, $J_4\ll\xi(\log 3q)^2$, both $o(T)$ uniformly for $q\le T^{A}$.

We relate $J_3$ to $J_1$ by the logarithmic-derivative functional equation of Lemma~\ref{lem:Xfactor}; note the conductor in the $\log\frac{q|t|}{2\pi}$ term, which is where $\log qT$ enters the main terms. Making the change of variables $s\mapsto1-s$ on the left edge (whose downward orientation reverses the sign of the $\log$-term) and taking complex conjugates, we obtain
\[
M_2=\overline{K}+2\Re J_1+O\bigl(\xi(\log qT)^2\bigr),
\]
where
\[
K=\frac{1}{2\pi}\int_1^T\log\frac{qt}{2\pi}\,\mathcal{L}_\chi(c+it)\,\mathcal{L}_{\bar\chi}(1-c-it)\,dt
\]
and where $J_1$ is the right-edge integral
\[
J_1=\frac{1}{2\pi}\int_1^T\frac{L'}{L}(c+it,\chi)\,\mathcal{L}_\chi(c+it)\,\mathcal{L}_{\bar\chi}(1-c-it)\,dt .
\]
We evaluate $K$ and $J_1$ in turn; here $K$ is real to leading order, so $\overline K$ and $K$ have the same main term.

We first evaluate $K$. Set $A(t)=\int_1^t\mathcal{L}_\chi(c+iu)\mathcal{L}_{\bar\chi}(1-c-iu)\,du$. The two factors are Dirichlet polynomials whose coefficient exponents sum to $1$, so Lemma~\ref{lem:MV} applies and, by Lemma~\ref{lem:densities},
\[
\begin{aligned}
A(t)&=t\sum_{\substack{n\le\xi\\(n,q)=1}}\frac{\mu(n)^2P_n^2}{n}+O\bigl(\xi(\log T)^{1/2}\bigr)\\
&=A_q\,t\log\xi\int_0^1 P^2+O\bigl(t\log\log 3q+\xi(\log T)^{1/2}\bigr).
\end{aligned}
\]
Integrating by parts,
\[
K=\frac{1}{2\pi}\Bigl[\log\tfrac{qt}{2\pi}\,A(t)\Bigr]_1^T-\frac{1}{2\pi}\int_1^T\frac{A(t)}{t}\,dt.
\]
The first term on the right-hand side contributes $\frac{1}{2\pi}A_q\,T\bigl(\log\tfrac{qT}{2\pi}\bigr)(\log\xi)\int_0^1 P^2$ to main order, while $\int_1^T\frac{A(t)}{t}dt=A_q\,T\log\xi\int_0^1 P^2+O(T\log\log 3q)$ is smaller by a full factor $\log qT$. Hence
\begin{align}
K&=\frac{A_q}{2\pi}\,T(\log\xi)(\log qT)\int_0^1 P^2+O\bigl(T\log T+T\log qT\log\log 3q\bigr)\notag\\
&=\frac{T}{2\pi}\,A_q\log qT\;\vartheta\log T\!\int_0^1 P^2+O\bigl(T\log T\log\log 3qT\bigr),\label{eq:K-MN}
\end{align}
using $\log qT\ll_A\log T$ in the error.

We now evaluate $J_1$. On the line $\Re s=c$ we expand
\[
-\frac{L'}{L}(s,\chi)\,\mathcal{L}_\chi(s)=\Bigl(\sum_{k}\Lambda(k)\chi(k)k^{-s}\Bigr)\Bigl(\sum_{l\le\xi}\mu(l)\chi(l)P_l\,l^{-s}\Bigr)=\sum_{n}\alpha_n n^{-s},
\]
where $\Lambda$ is the von Mangoldt function and, by complete multiplicativity of $\chi$, the coefficients factor as $\alpha_n=\chi(n)\,\widehat\alpha_n$ with
\[
\widehat\alpha_n=\sum_{\substack{kl=n\\ l\le\xi}}\Lambda(k)\mu(l)P_l .
\]
Applying Lemma~\ref{lem:MV} with this Dirichlet series against $\mathcal{L}_{\bar\chi}(1-s)$ on $\Re s=c$, the diagonal pairing $\chi(n)\overline{\chi(n)}=\mathbf 1_{(n,q)=1}$ forces $(n,q)=1$ and leaves
\[
J_1=-\frac{T}{2\pi}\sum_{\substack{n\le\xi\\(n,q)=1}}\frac{\widehat\alpha_n\,\mu(n)P_n}{n}+O\bigl(\xi(\log T)^{3/2}\bigr).
\]
We now evaluate the main sum. Since $\mu(n)\ne0$ requires $n$ square-free, in $\sum_{kl=n}\Lambda(k)\mu(l)\,\mu(n)$ the factor $\Lambda(k)$ restricts $k$ to a single prime $p$ (a higher prime power would make $n=kl$ non-square-free), and then $p\nmid l$ and $\mu(pl)=-\mu(l)$, so
\[
\Lambda(p)\,\mu(l)\,\mu(pl)=-\mu(l)^2\log p .
\]
Carrying out the sum over $k$ first, and writing $a=\frac{\log(\xi/l)}{\log\xi}$ so that $P_l=P(a)$, only the primes $k=p$ contribute. The prime number theorem in the form
\[
\sum_{\substack{p\le y\\ p\nmid lq}}\frac{\log p}{p}=\log y+O(\log\log 3lq),
\]
where the cost of removing the primes dividing $lq$ is controlled by Lemma~\ref{lem:densities}, together with partial summation, then yields
\[
\sum_{\substack{p\le\xi/l\\ p\nmid lq}}\frac{\log p}{p}\,P_{pl}=\log\xi\int_0^a P(u)\,du+O(\log\log 3qT).
\]
Hence the main sum equals $-\sum_{l\le\xi,\,(l,q)=1}\frac{\mu(l)^2 P_l}{l}\Bigl(\log\xi\int_0^a P\Bigr)+O\bigl(\log\xi\log\log 3qT\bigr)$, and summing over $l$ by Lemma~\ref{lem:densities} once more,
\[
\sum_{\substack{n\le\xi\\(n,q)=1}}\frac{\widehat\alpha_n\,\mu(n)P_n}{n}
=-A_q(\log\xi)^2\int_0^1 P(a)\Bigl(\int_0^a P(u)\,du\Bigr)da+O\bigl(\log\xi\log\log 3qT\bigr).
\]
By the substitution $Q(a)=\int_0^a P$ the inner double integral satisfies
\[
\int_0^1 P(a)\Bigl(\int_0^a P(u)\,du\Bigr)da=\int_0^1 Q'(a)Q(a)\,da=\tfrac12 Q(1)^2=\tfrac12\Bigl(\int_0^1 P\Bigr)^2,
\]
so the main sum is $-\tfrac{A_q}{2}(\log\xi)^2\bigl(\int_0^1 P\bigr)^2+O(\log\xi\log\log 3qT)$. Substituting into the expression for $J_1$ and recalling $\log\xi=\vartheta\log T$,
\begin{equation}\label{eq:J1-MN}
2\Re J_1=\frac{T}{2\pi}\,A_q\log T\;\vartheta^2\log T\Bigl(\int_0^1 P\Bigr)^2+O\bigl(T\log T\log\log 3qT\bigr).
\end{equation}
Adding \eqref{eq:K-MN} and \eqref{eq:J1-MN} proves Proposition~\ref{prop:M2MN}.

\section{Proof of Propositions \ref{prop:M1Ng} and \ref{prop:M2Ng}}\label{sect:propsNg}

This section runs parallel to Section~\ref{sect:propsMN}, and the reader will recognise the overall structure at once. We nevertheless give the details in full: the changes, where they occur, are subtle, and one step, the left edge of $\widetilde M_1$, is genuinely different. Throughout, the mollifier is $\widetilde{\mathcal{L}}_\psi(s)=\sum_{n\le\xi}\lambda(n)\psi(n)P_n n^{-s}$, and the underlying Dirichlet series is $L(2s,\chi^2)/L(s,\chi)=\sum_n\lambda(n)\chi(n)n^{-s}$ from \eqref{eq:liouville-id}. The arithmetic factor changes from $A_q$ to $\varphi(q)/q$ because the diagonal coefficient is now $|\lambda(n)\chi(n)|^2=\mathbf 1_{(n,q)=1}$. It is also in this section that the restriction $q\le T^{1-\delta}$ of Theorem~\ref{thm:Ng} arises, at the horizontal edges of the first moment; we discuss it, and how GRH for $L(s,\chi^2)$ would remove it, in the remark following that estimate.

\subsection{The first moment $\widetilde M_1$}
Let $T\in\mathcal T$, let $t_0$ be the bottom height of Lemma~\ref{lem:heights}, and let $3\le q\le T^{1-\delta}$. As $L(2s,\chi^2)/L(s,\chi)$ has a simple pole at each zero $\rho$ of residue $L(2\rho,\chi^2)/L'(\rho,\chi)$, Cauchy's residue theorem gives, over the same rectangle $\mathcal R$,
\[
\widetilde M_1=\frac{1}{2\pi i}\oint_{\mathcal R}\frac{L(2s,\chi^2)}{L(s,\chi)}\,\widetilde{\mathcal{L}}_{\bar\chi}(1-s)\,ds=I_1+I_2+I_3+I_4 ,
\]
say. (When $\chi$ is real, $\chi^2$ is the principal character and $L(2s,\chi^2)=\zeta(2s)\prod_{p\mid q}(1-p^{-2s})$ has a pole at $s=\tfrac12$; this lies at height $t=0$, below $\mathcal R$, and so is not enclosed, exactly as for $\zeta(2s)$ in the case $q=1$. When $\chi$ is complex, $L(2s,\chi^2)$ is entire.)

We begin with the right edge $I_1$ (as in Section~\ref{sect:propsMN}, the $t$-integration is silently extended from $t_0$ down to $1$, an $O(\xi(qT)^{o(1)})=o(T)$ change). On $\Re s=c$ both factors are absolutely convergent, $L(2s,\chi^2)/L(s,\chi)=\sum_m\lambda(m)\chi(m)m^{-s}$ and $\widetilde{\mathcal{L}}_{\bar\chi}(1-s)=\sum_{n\le\xi}\lambda(n)\bar\chi(n)P_n n^{s-1}$, so with $s=c+it$,
\[
I_1=\frac{1}{2\pi}\int_1^T\Bigl(\sum_m\frac{\lambda(m)\chi(m)}{m^{c+it}}\Bigr)\Bigl(\sum_{n\le\xi}\frac{\lambda(n)\bar\chi(n)P_n}{n^{1-c-it}}\Bigr)\,dt.
\]
Applying Lemma~\ref{lem:MV} with $a_m=\lambda(m)\chi(m)m^{-c}$ and $b_n=\lambda(n)\bar\chi(n)P_n n^{c-1}$,
\[
I_1=\frac{T}{2\pi}\sum_{n\le\xi}\lambda(n)^2|\chi(n)|^2\frac{P_n}{n}+O\!\left(\Bigl(\sum_m m|a_m|^2\Bigr)^{1/2}\Bigl(\sum_{n\le\xi}n|b_n|^2\Bigr)^{1/2}\right).
\]
Since $\lambda(n)^2=1$, the main sum is $\sum_{n\le\xi,(n,q)=1}P_n/n$. For the error,
\[
\sum_m m|a_m|^2=\sum_m|\chi(m)|^2 m^{1-2c}\le\zeta(2c-1)\ll\log T
\]
and
\[
\sum_{n\le\xi}n|b_n|^2=\sum_{n\le\xi}|\chi(n)|^2P_n^2\,n^{2c-1}\ll\xi^{2c}\ll\xi^2,
\]
with absolute constants, so the error is $O\bigl(\xi(\log T)^{1/2}\bigr)$. By the first density of Lemma~\ref{lem:densities} and partial summation,
\[
\sum_{\substack{n\le\xi\\(n,q)=1}}\frac{P_n}{n}=\frac{\varphi(q)}{q}\log\xi\int_0^1 P+O(\log\log 3q),
\]
hence
\begin{align}
I_1&=\frac{T}{2\pi}\frac{\varphi(q)}{q}(\log\xi)\int_0^1 P+O\bigl(\xi(\log T)^{1/2}+T\log\log 3q\bigr)\notag\\
&=\frac{\varphi(q)}{q}\,\frac{T}{2\pi}\log T\;\vartheta\!\int_0^1 P+O\bigl(T\log\log 3q\bigr).\label{eq:I1-Ng}
\end{align}

Next we bound the horizontal edges. On $I_2$ we now have, in addition to the factor $|L(\sigma+iT,\chi)|^{-1}\ll(qT)^{o(1)}$ from the $-\tfrac12\le\sigma\le2$ clause of Lemma~\ref{lem:GRHsize}, the numerator $|L(2\sigma+2iT,\chi^2)|$. For $\sigma\ge\tfrac12$ this is $\ll(qT)^{o(1)}$ since $2\sigma\ge1$ (the imprimitive Euler factors contributing at most $2^{\omega(q)}$), while for $1-c\le\sigma<\tfrac12$ the convexity bound for the primitive character inducing $\chi^2$ (whose conductor divides $q$), together with the boundedness of the finitely many imprimitive Euler factors, gives $|L(2\sigma+2iT,\chi^2)|\ll(qT)^{1/2-\sigma+o(1)}$. The integrand is therefore $\ll(qT)^{o(1)}\bigl(\xi+(qT)^{1/2}\bigr)$, and
\[
I_2\ll \xi\,(qT)^{o(1)}+(qT)^{1/2+o(1)}=o(T)
\qquad\text{precisely when }q\le T^{1-\delta}.
\]
On the bottom edge, at the admissible height $t_0$, convexity gives $|L(2\sigma+2it_0,\chi^2)|\ll q^{1/2-\sigma+o(1)}$ for $\sigma<\tfrac12$ and $|L(\sigma+it_0,\chi)|^{-1}\ll q^{o(1)}$, so
\[
I_4\ll\bigl(\xi+q^{1/2}\bigr)(qT)^{o(1)}=o(T)
\]
for $q\le T^{1-\delta}$.

\begin{remark}
The estimate for $I_2$ is the only point in the paper at which the range of $q$ is restricted; every other estimate in this section is uniform for $q\le T^{A}$. The restriction enters because we know only the convexity bound for $L(\cdot,\chi^2)$ in the strip $0<\Re w<1$. If one assumes GRH for $L(s,\chi^2)$, the restriction can be removed. Keeping the left edge at $\Re s=1-c$ does not help: the worst horizontal contribution is near $\sigma=1-c$, where $2\sigma\approx0$ and even GRH gives $|L(2\sigma+2iT,\chi^2)|\asymp(qT)^{1/2}$, the same size as convexity. The efficient route is instead to move the left edge of the rectangle to $\Re s=\tfrac14$, so that $2s$ stays on or to the right of the critical line of $L(\cdot,\chi^2)$ and the Lindel\"of-type bound of Lemma~\ref{lem:GRHsize} applies to it. The horizontal edges are then $\ll\xi(qT)^{o(1)}$, the new left edge is $\ll q^{-1/4}T^{3/4}\xi^{1/4}(qT)^{o(1)}=o(T)$, and the range extends to $q\le T^{A}$ as in Theorem~\ref{thm:MN}. This is also the route taken by Sinha \cite{SinhaThesis} at $q=1$, and it renders the oscillation argument for the left edge below unnecessary. We prefer to keep Theorem~\ref{thm:Ng} unconditional in $\chi^2$, which forces the left edge to stay at $\Re s=1-c$ and is the reason for the functional-equation reflection that follows.
\end{remark}

It remains to bound the left edge $I_3$, and here the argument departs from Section~\ref{sect:propsMN}. On $\Re s=1-c$ we have $2\Re s=2-2c=-2/\log T<0$, so $2s$ lies just to the left of the critical strip of $L(\cdot,\chi^2)$, and by the functional equation $|L(2s,\chi^2)|\ll(q|t|)^{1/2}T^{o(1)}$, a bound of genuinely growing size. This growth offsets the decay $|L(s,\chi)|^{-1}\ll(q|t|)^{-1/2}\log T$ established in Section~\ref{sect:propsMN}: the product of the two bounds no longer decays in $t$, and the resulting trivial estimate $I_3\ll T(\log T)^3$ exceeds the main term. We therefore exploit the oscillation of the archimedean factor instead.

Let $\chi^*$ be the primitive character, of conductor $q^*\mid q$, that induces $\chi^2$, so that
\[
L(2s,\chi^2)=E(2s)\,L(2s,\chi^*),
\qquad
E(w)=\prod_{\substack{p\mid q\\ p\nmid q^*}}\bigl(1-\chi^*(p)p^{-w}\bigr),
\]
the factor $E$ being a finite Euler product. Applying the functional equations $L(2s,\chi^*)=X(2s,\chi^*)L(1-2s,\bar\chi^*)$ and $L(s,\chi)=X(s,\chi)L(1-s,\bar\chi)$, we may write, for $s=(1-c)+it$ with $1\le t\le T$,
\begin{equation}\label{eq:reflection}
\frac{L(2s,\chi^2)}{L(s,\chi)}=\underbrace{\frac{X(2s,\chi^*)}{X(s,\chi)}}_{G(t)}\;E(2s)\;\frac{L(1-2s,\bar\chi^*)}{L(1-s,\bar\chi)},
\end{equation}
where $G(t)$ denotes the quotient of the two archimedean factors. The point of \eqref{eq:reflection} is that $G(t)$ is a ratio of two primitive archimedean factors, while the remaining factors are evaluated in the region of absolute convergence: indeed $\Re(1-2s)=2c-1>1$ and $\Re(1-s)=c>1$, while $E(2s)=\sum_d c_d\,d^{-2s}$ is a Dirichlet polynomial supported on square-free $d\mid\prod_{p\mid q}p$ with $|c_d|\le1$. Writing $s=(1-c)+it$, these factors expand as
\[
\begin{gathered}
L(1-2s,\bar\chi^*)=\sum_{a\ge1}\frac{\overline{\chi^*}(a)}{a^{2c-1}}\,(a^2)^{it},\qquad
\frac{1}{L(1-s,\bar\chi)}=\sum_{b\ge1}\frac{\mu(b)\bar\chi(b)}{b^{c}}\,b^{it},\\
\widetilde{\mathcal{L}}_{\bar\chi}(1-s)=\sum_{m\le\xi}\frac{\lambda(m)\bar\chi(m)P_m}{m^{c}}\,m^{it},
\end{gathered}
\]
together with $E(2s)=\sum_d c_d\,d^{2(c-1)}\,(d^{2})^{-it}$, the last carrying negative frequencies. Their product is therefore an absolutely convergent sum of oscillations over the positive rational frequencies $r=a^2bm/d^2$,
\[
E(2s)\,\frac{L(1-2s,\bar\chi^*)}{L(1-s,\bar\chi)}\,\widetilde{\mathcal{L}}_{\bar\chi}(1-s)
=\sum_{\substack{a,d,b\ge1\\ m\le\xi}}\beta_{a,d,b,m}\,r^{it},
\]
with coefficients
\[
\beta_{a,d,b,m}=c_d\,d^{2(c-1)}\,\frac{\overline{\chi^*}(a)}{a^{2c-1}}\,\frac{\mu(b)\bar\chi(b)}{b^{c}}\,\frac{\lambda(m)\bar\chi(m)P_m}{m^{c}},
\]
so that
\[
I_3=-\frac{1}{2\pi}\sum_{a,d,b,m}\beta_{a,d,b,m}\,J_r,
\qquad
J_r=\int_1^T G(t)\,r^{it}\,dt .
\]

We bound each $J_r$ uniformly in $r$ and in $q$ by Lemma~\ref{lem:vdc}. By Lemma~\ref{lem:Xfactor}, the modulus
\[
|G(t)|=\frac{|X(2-2c+2it,\chi^*)|}{|X(1-c+it,\chi)|}\asymp\Bigl(\frac{q^*}{q}\Bigr)^{1/2}\bigl(qq^*t\bigr)^{O(1/\log T)}\ll e^{O(A)}\ll1
\]
for $q\le T^{A}$ and $1\le t\le T$, with monotone modulus up to $O(1/t)$, so $\sup|G|+\operatorname{Var}G\ll1$ with constants depending only on $A$. For the phase, write $\psi_r(t)=\arg G(t)+t\log r$; the term $t\log r$ is linear in $t$ whatever the (rational) frequency $r>0$, so Lemma~\ref{lem:vdc} applies just as for integer frequencies. By the differentiated Stirling asymptotic of Lemma~\ref{lem:Xfactor}, in which the conductors of $\chi$ and $\chi^*$ enter only the $t$-linear part of the phase and hence shift $\psi_r'$ but not $\psi_r''$,
\[
\psi_r'(t)=-\log t+\log r+c_{q,q^*}+O\!\left(\frac1t\right),
\qquad
\psi_r''(t)=-\frac1t+O\!\left(\frac1{t^2}\right),
\]
where $c_{q,q^*}=\log\frac{q}{2\pi}-2\log\frac{q^*}{\pi}\ll\log q$ is constant in $t$ (it shifts the stationary point but not the curvature). The second derivative is independent of $r$ and of $q$ and, for $t\ge t_1$ with $t_1$ an absolute constant, is of constant sign with $|\psi_r''(t)|\ge T^{-1}(1+o(1))$ on $[t_1,T]$; the initial segment $[1,t_1]$ contributes $O(1)$ to $J_r$ trivially. Lemma~\ref{lem:vdc} therefore gives, uniformly in $r$ and $q$,
\[
|J_r|\ll\Bigl(\min_{[t_1,T]}|\psi_r''|\Bigr)^{-1/2}\ll\sqrt T .
\]
It remains to bound the coefficient sum. The imprimitivity polynomial contributes
\[
\sum_d|c_d|\,d^{2(c-1)}\le2^{\omega(q)}q^{2/\log T}=T^{o(1)}
\qquad(q\le T^{A}),
\]
and hence
\[
\sum_{a,d,b,m}|\beta_{a,d,b,m}|
\ll T^{o(1)}\Bigl(\sum_a \frac{1}{a^{2c-1}}\Bigr)\Bigl(\sum_b\frac{\mu(b)^2}{b^{c}}\Bigr)\Bigl(\max_{[0,1]}|P|\sum_{m\le\xi}\frac{1}{m^{c}}\Bigr)
\le T^{o(1)}\,\zeta(2c-1)\,\frac{\zeta(c)^2}{\zeta(2c)}\max_{[0,1]}|P| ,
\]
which is $\ll T^{o(1)}(\log T)^3$ since $\zeta(2c-1),\zeta(c)\ll\log T$ and $\zeta(2c)\asymp1$. Combining the last three displays,
\begin{equation}\label{eq:I3-Ng}
I_3\ll\Bigl(\sum_{a,d,b,m}|\beta_{a,d,b,m}|\Bigr)\,\sup_{r}|J_r|\ll T^{1/2+o(1)}=o(T).
\end{equation}
No hypothesis on $L(s,\chi^2)$ has been used: the reflected $L$-factors in \eqref{eq:reflection} are evaluated only in $\Re>1$, and $E$ is a finite Euler product. Combining \eqref{eq:I1-Ng}, the horizontal bounds, and \eqref{eq:I3-Ng} proves Proposition~\ref{prop:M1Ng}.

\subsection{The second moment $\widetilde M_2$}
By the symmetry \eqref{eq:symm}, $\widetilde M_2=\sum_{t_0<\gamma\le T}\widetilde{\mathcal L}_\chi(\rho)\widetilde{\mathcal L}_{\bar\chi}(1-\rho)$, and since the residue of $\frac{L'}{L}(s,\chi)$ at a simple zero is $1$,
\[
\widetilde M_2=\frac{1}{2\pi i}\oint_{\mathcal R}\frac{L'}{L}(s,\chi)\,\widetilde{\mathcal L}_\chi(s)\,\widetilde{\mathcal L}_{\bar\chi}(1-s)\,ds=\widetilde J_1+\widetilde J_2+\widetilde J_3+\widetilde J_4,
\]
say. Taking $T$ in the sequence of Lemma~\ref{lem:heights}, the trivial bound $\widetilde{\mathcal L}_\psi(\sigma+it)\ll\xi^{1-\sigma}$ together with $\frac{L'}{L}(\sigma+iT,\chi)\ll(\log qT)^2$ gives $\widetilde J_2\ll\xi(\log qT)^2$ and, at the bottom height $t_0$, $\widetilde J_4\ll\xi(\log 3q)^2$, both $o(T)$ uniformly for $q\le T^{A}$. Relating $\widetilde J_3$ to $\widetilde J_1$ through the logarithmic-derivative functional equation of Lemma~\ref{lem:Xfactor}, with the change of variables $s\mapsto1-s$ on the left edge (whose downward orientation reverses the sign of the $\log$-term) and taking complex conjugates,
\[
\widetilde M_2=\overline{\widetilde K}+2\Re\widetilde J_1+O\bigl(\xi(\log qT)^2\bigr),
\]
where
\[
\widetilde K=\frac{1}{2\pi}\int_1^T\log\frac{qt}{2\pi}\,\widetilde{\mathcal L}_\chi(c+it)\,\widetilde{\mathcal L}_{\bar\chi}(1-c-it)\,dt
\]
and where $\widetilde J_1$ is the right-edge integral
\[
\widetilde J_1=\frac{1}{2\pi}\int_1^T\frac{L'}{L}(c+it,\chi)\,\widetilde{\mathcal L}_\chi(c+it)\,\widetilde{\mathcal L}_{\bar\chi}(1-c-it)\,dt .
\]

We first evaluate $\widetilde K$ (real to leading order, so $\overline{\widetilde K}$ and $\widetilde K$ have the same main term). Set $\widetilde A(t)=\int_1^t\widetilde{\mathcal L}_\chi(c+iu)\widetilde{\mathcal L}_{\bar\chi}(1-c-iu)\,du$. The two factors have coefficient exponents summing to $1$, and the diagonal coefficient is $\lambda(n)^2|\chi(n)|^2=\mathbf 1_{(n,q)=1}$, so by Lemma~\ref{lem:MV} and Lemma~\ref{lem:densities},
\[
\begin{aligned}
\widetilde A(t)&=t\sum_{\substack{n\le\xi\\(n,q)=1}}\frac{P_n^2}{n}+O\bigl(\xi(\log T)^{1/2}\bigr)\\
&=\frac{\varphi(q)}{q}\,t\log\xi\int_0^1 P^2+O\bigl(t\log\log 3q+\xi(\log T)^{1/2}\bigr).
\end{aligned}
\]
Integrating by parts,
\[
\widetilde K=\frac{1}{2\pi}\Bigl[\log\tfrac{qt}{2\pi}\,\widetilde A(t)\Bigr]_1^T-\frac{1}{2\pi}\int_1^T\frac{\widetilde A(t)}{t}\,dt,
\]
and since $\int_1^T\frac{\widetilde A(t)}{t}dt=\tfrac{\varphi(q)}{q}T\log\xi\int_0^1 P^2+O(T\log\log 3q)$ is smaller than the first term on the right-hand side by a full factor $\log qT$,
\begin{align}
\widetilde K&=\frac{\varphi(q)}{q}\,\frac{T(\log\xi)(\log qT)}{2\pi}\int_0^1 P^2+O\bigl(T\log T+T\log qT\log\log 3q\bigr)\notag\\
&=\frac{\varphi(q)}{q}\,\frac{T}{2\pi}\log qT\;\vartheta\log T\!\int_0^1 P^2+O\bigl(T\log T\log\log 3qT\bigr).\label{eq:K-Ng}
\end{align}

We now evaluate $\widetilde J_1$. On $\Re s=c$ we expand
\[
-\frac{L'}{L}(s,\chi)\,\widetilde{\mathcal L}_\chi(s)=\Bigl(\sum_k\Lambda(k)\chi(k)k^{-s}\Bigr)\Bigl(\sum_{l\le\xi}\lambda(l)\chi(l)P_l\,l^{-s}\Bigr)=\sum_n\widetilde\alpha_n\chi(n)\,n^{-s},
\]
where $\widetilde\alpha_n=\sum_{kl=n,\,l\le\xi}\Lambda(k)\lambda(l)P_l$ and we have used $\chi(k)\chi(l)=\chi(n)$. Pairing this against $\widetilde{\mathcal L}_{\bar\chi}(1-s)=\sum_n\lambda(n)\bar\chi(n)P_n n^{s-1}$ and applying Lemma~\ref{lem:MV}, the diagonal pairing $\chi(n)\overline{\chi(n)}=\mathbf 1_{(n,q)=1}$ forces $(n,q)=1$ and leaves
\[
\widetilde J_1=-\frac{T}{2\pi}\sum_{\substack{n\le\xi\\(n,q)=1}}\frac{\widetilde\alpha_n\,\lambda(n)P_n}{n}+O\bigl(\xi(\log T)^{3/2}\bigr).
\]
To evaluate the main sum we substitute $\widetilde\alpha_n$ and use the complete multiplicativity of the Liouville function, $\lambda(l)\lambda(kl)=\lambda(l)^2\lambda(k)=\lambda(k)$, to obtain
\[
\begin{aligned}
\sum_{\substack{n\le\xi\\(n,q)=1}}\frac{\widetilde\alpha_n\,\lambda(n)P_n}{n}
&=\sum_{\substack{kl\le\xi\\(kl,q)=1}}\frac{\Lambda(k)\lambda(l)P_l\,\lambda(kl)P_{kl}}{kl}\\
&=\sum_{\substack{kl\le\xi\\(kl,q)=1}}\frac{\Lambda(k)\lambda(k)\,P_l P_{kl}}{kl}.
\end{aligned}
\]
The factor $\Lambda(k)\lambda(k)$ is supported on prime powers, with $\Lambda(p)\lambda(p)=-\log p$ on primes (and $\Lambda(p^j)\lambda(p^j)=(-1)^j\log p$ for $j\ge2$ contributing $O(\log\xi)$ in total, absorbed into the error term), and vanishes for $p\mid q$. Writing $a=\frac{\log(\xi/l)}{\log\xi}$ and carrying out the sum over $k$ first, the prime number theorem in the form
\[
\sum_{\substack{p\le y\\ p\nmid q}}\frac{\log p}{p}=\log y+O(\log\log 3q)
\]
and partial summation give
\[
\sum_{\substack{k\le\xi/l\\(k,q)=1}}\frac{\Lambda(k)\lambda(k)}{k}\,P_{kl}=-\log\xi\int_0^a P(u)\,du+O(\log\log 3q).
\]
Summing over $l\le\xi$ with $(l,q)=1$ by Lemma~\ref{lem:densities}, and writing $Q(a)=\int_0^a P$,
\[
\begin{aligned}
\sum_{\substack{n\le\xi\\(n,q)=1}}\frac{\widetilde\alpha_n\,\lambda(n)P_n}{n}
&=-\frac{\varphi(q)}{q}(\log\xi)^2\int_0^1 P(a)Q(a)\,da+O\bigl(\log\xi\log\log 3qT\bigr)\\
&=-\frac{\varphi(q)}{2q}(\log\xi)^2\Bigl(\int_0^1 P\Bigr)^2+O\bigl(\log\xi\log\log 3qT\bigr),
\end{aligned}
\]
since
\[
\int_0^1 P(a)Q(a)\,da=\int_0^1 Q'(a)Q(a)\,da=\tfrac12 Q(1)^2=\tfrac12\Bigl(\int_0^1 P\Bigr)^2.
\]
Substituting into the expression for $\widetilde J_1$ and recalling $\log\xi=\vartheta\log T$,
\begin{equation}\label{eq:J1-Ng}
2\Re\widetilde J_1=\frac{\varphi(q)}{q}\,\frac{T}{2\pi}\log T\;\vartheta^2\log T\Bigl(\int_0^1 P\Bigr)^2+O\bigl(T\log T\log\log 3qT\bigr).
\end{equation}
Adding \eqref{eq:K-Ng} and \eqref{eq:J1-Ng} proves Proposition~\ref{prop:M2Ng}.

\section{Optimisation of the mollifier, the conductor deficit, and the conjectures}\label{sect:opt}

\subsection{Proof of Theorems \ref{thm:MN} and \ref{thm:Ng}}
We treat both theorems at once; write $\mathfrak a=\mathfrak a(q)$ for the arithmetic factor \eqref{eq:aq}, and let $(M_1,M_2)$ stand for the relevant pair from \eqref{eq:CS-MN} or \eqref{eq:CS-Ng}, with $q$ in the relevant range.

Fix $T\ge4$ and choose a height $\tau\in\mathcal T$ with $T-2\le\tau<T$, possible since the gaps in $\mathcal T$ are bounded. As the summand in \eqref{eq:CS-MN} (respectively \eqref{eq:CS-Ng}) is non-negative, the moment up to $T$ dominates the moment up to $\tau$, so it suffices to prove the bound at $\tau$. By Propositions~\ref{prop:M1MN} and~\ref{prop:M2MN} (respectively~\ref{prop:M1Ng} and~\ref{prop:M2Ng}), and since $\mathfrak a(q)\gg1/\log\log q$ makes every error term smaller than its main term by the uniform relative factor $O\bigl((\log\log qT)^2/\log T\bigr)$,
\[
\frac{|M_1|^2}{M_2}
=\frac{\Bigl(\frac{T}{2\pi}\mathfrak a\log T\;\vartheta\!\int_0^1 P\Bigr)^2}{\frac{T}{2\pi}\mathfrak a\log T\Bigl[\log qT\;\vartheta\!\int_0^1 P^2+\log T\;\vartheta^2\bigl(\int_0^1 P\bigr)^2\Bigr]}\;\bigl(1+o(1)\bigr),
\]
with the $o(1)$ uniform in the stated ranges of $q$. Cancelling the common factors,
\begin{equation}\label{eq:master}
\frac{|M_1|^2}{M_2}=\frac{\mathfrak a(q)}{2\pi}\,T\,\frac{ur}{1+ur}\,\bigl(1+o(1)\bigr),
\qquad
u=\frac{\log\xi}{\log qT}=\vartheta\beta,
\quad
r=\frac{\bigl(\int_0^1 P\bigr)^2}{\int_0^1 P^2},
\end{equation}
which is increasing in both $u$ and $r$. Equation \eqref{eq:master} is the key formula of the paper: every statement in Section~\ref{sect:intro} is a reading of it.

The two variables are constrained as follows. By the Cauchy--Schwarz inequality $\bigl(\int_0^1 P\bigr)^2\le\int_0^11^2\,\int_0^1P^2$, so $r\le1$. The constraints $P(0)=0$ and $P(1)=1$ exclude the constant polynomial, so $r=1$ is a supremum rather than a maximum; it is approached by the admissible choice $P(x)=1-(1-x)^N$ as $N\to\infty$, for which $\int_0^1P=\tfrac{N}{N+1}$ and $\int_0^1P^2=\tfrac{2N^2}{(N+1)(2N+1)}$, so that $r\to1$.

As for $u$, we have $u\le\beta$: the mean-value error terms of Lemma~\ref{lem:MV} cap the mollifier at $\xi\le T^{1-\varepsilon}$ independently of $q$ (the remark following the lemma), that is, $\vartheta<1$.

Hence, given $\varepsilon>0$, taking $N$ large and then $\vartheta$ close enough to $1$,
\[
\sum_{0<\gamma\le T}|R(\rho)|^2\ge\frac{\mathfrak a(q)}{2\pi}\,\frac{ur}{1+ur}\,T\,\bigl(1+o(1)\bigr)\ge(1-\varepsilon)\,\frac{\mathfrak a(q)}{2\pi}\,\frac{\beta}{1+\beta}\,T
\]
for all sufficiently large $T$, uniformly in the stated ranges of $q$, where $R(\rho)$ is the relevant summand and we used that $v\mapsto v/(1+v)$ is increasing with $\tfrac{\vartheta\beta r}{1+\vartheta\beta r}\to\tfrac{\beta}{1+\beta}$ as $\vartheta,r\to1$, uniformly in $\beta\in(0,1)$.

With $\mathfrak a=A_q$ this is Theorem~\ref{thm:MN}, and with $\mathfrak a=\varphi(q)/q$ it is Theorem~\ref{thm:Ng}. the small-conductor estimates follow since $\log q=o(\log T)$ gives $\beta\to1$ and $\tfrac{\beta}{1+\beta}\to\tfrac12$, and the power-conductor estimate follows since $q=T^{A}$ gives $\beta=\tfrac1{1+A}$ and $\tfrac{\beta}{1+\beta}=\tfrac{1}{2+A}$. \qed

\subsection{The conductor deficit and the conjectures}\label{subsect:deficit}
The key formula \eqref{eq:master} tells the whole story at once. The method produces the quantity $\tfrac{\mathfrak a(q)}{2\pi}\tfrac{u}{1+u}T$ at $r\to1$; the mean value theorem caps $u$ at $\beta=\log T/\log qT$; and the cap is conductor-independent because the $t$-integration window $[1,T]$, not the conductor, sets the resolution of the mean value theorem: a single $L$-function observed up to height $T$ cannot support a mollifier longer than $T$, while the zero density it must match grows with $\log qT$. For small conductors $\beta\to1$ and the proved value is half the limiting one; for $q=T^A$ the proportion drops to $\tfrac1{2+A}$.

In the opposite direction, Farmer's long-mollifier heuristic \cite{Farmer} corresponds to letting $u\to\infty$ in \eqref{eq:master}: the formal limit is $\tfrac{\mathfrak a(q)}{2\pi}T$, and this is the content of Conjectures~\ref{conj:dirMN} and~\ref{conj:dirNg}. In particular our theorems are, conjecturally, exactly the proportion $\tfrac{\beta}{1+\beta}$ of the truth, one half in the small-conductor regime, the deficiency being the ratio of the value of $u\mapsto\tfrac{u}{1+u}$ at the cap $u=\beta$ to its limiting value $1$, a structural feature of the diagonal Cauchy--Schwarz construction rather than a numerical accident.

This heuristic is, of course, not a proof, and it is applied here to a discrete sum over zeros rather than to the continuous mean values for which it was framed. Its main external support is the case $q=1$, where the limiting values $\tfrac{3}{\pi^3}T$ and $\tfrac{T}{2\pi}$ coincide with the independently-formulated Conjectures~\ref{conj:gonek} and~\ref{conj:ng} of Gonek and Ng. As noted in Section~\ref{sect:intro}, it would be of interest to see whether Conjectures~\ref{conj:dirMN} and~\ref{conj:dirNg} can be derived from the recipe of Conrey--Farmer--Keating--Rubinstein--Snaith \cite{CFKRS} or from the $L$-functions Ratios Conjecture \cite{CFZ,ConreySnaith}.

\section{Proof of Corollaries in Section \ref{sec:higher}}\label{sec:higher-proofs}

The four corollaries we prove here are elementary consequences of the second-moment lower bounds of Theorems~\ref{thm:MN} and~\ref{thm:Ng}: each follows from a single application of H\"older's inequality, fed only by the Riemann--von Mangoldt count of the zeros of $L(s,\chi)$, and uses no further information about the zeros or the $L$-function. We first prove the single-character bounds, Corollaries~\ref{cor:MN-higher} and~\ref{cor:Ng-higher}; the family forms then follow by summing over the primitive characters modulo $q$.

\begin{proof}[Proof of Corollaries~\textup{\ref{cor:MN-higher}} and \textup{\ref{cor:Ng-higher}}]
Write $R(\rho)$ for the summand in each corollary and $\mathfrak a(q)$ for the arithmetic factor, so $\mathfrak a(q)=A_q$ in the reciprocal case and $\mathfrak a(q)=\varphi(q)/q$ in the ratio case. For $k=1$ the statements are Theorems~\ref{thm:MN} and~\ref{thm:Ng}, so assume $k>1$. By H\"older's inequality with exponents $k$ and $k/(k-1)$,
\[
  \sum_{0<\gamma\le T}|R(\rho)|^{2}
  \;\le\;\Bigl(\sum_{0<\gamma\le T}|R(\rho)|^{2k}\Bigr)^{1/k}N(T,\chi)^{1-1/k},
\]
where $N(T,\chi)=\frac{T}{2\pi}\log\frac{qT}{2\pi}-\frac{T}{2\pi}+O(\log qT)$ is
the zero-counting function for $L(s,\chi)$ (see Davenport \cite[Ch.~16]{Davenport}), uniformly for $q\le T^{A}$ this is $\frac{T}{2\pi}\log qT\,(1+o(1))$. Rearranging and inserting the relevant second-moment bound (Theorem~\ref{thm:MN}, resp.~\ref{thm:Ng}) for
the numerator,
\[
  \sum_{0<\gamma\le T}|R(\rho)|^{2k}
  \;\ge\;\frac{\bigl(\sum_{0<\gamma\le T}|R(\rho)|^{2}\bigr)^{k}}{N(T,\chi)^{\,k-1}}
  \;\ge\;\frac{\bigl((1-\varepsilon)\tfrac{\mathfrak a(q)}{2\pi}\tfrac{\beta}{1+\beta}T\bigr)^{k}}
            {\bigl(\tfrac{T}{2\pi}\log qT\bigr)^{k-1}}\,(1+o(1)).
\]
Since $\bigl(\tfrac{\mathfrak a(q)}{2\pi}\tfrac{\beta}{1+\beta}\bigr)^{k}\big/
\bigl(\tfrac1{2\pi}\bigr)^{k-1}
=\tfrac{\mathfrak a(q)^{k}}{2\pi}\bigl(\tfrac{\beta}{1+\beta}\bigr)^{k}$, this gives the first display of each corollary after absorbing $(1-\varepsilon)^{k}(1+o(1))$ into a single $1-\varepsilon$. The explicit form follows from $\tfrac{\beta}{1+\beta}=\bigl(2+\tfrac{\log q}{\log T}\bigr)^{-1}$. When $\log q=o(\log T)$ one has $\tfrac{\beta}{1+\beta}\to\tfrac12$ and $\log qT\sim\log T$, yielding the small-conductor displays.
\end{proof}

At $q=1$ these recover the H\"older bounds for the zeta-function:
$A_1=6/\pi^2$ gives, at $k=1$, the constant $A_1/4\pi=3/2\pi^3$ for the reciprocal moment, while $\varphi(1)/1=1$ gives $\tfrac{1}{2^{k+1}\pi}\,T/(\log T)^{k-1}$ for the ratio
moment.

The family forms require no further work: the bounds just proved are uniform in $\chi$ and depend on it only through the modulus $q$, so they may be summed over the $\varphi^{*}(q)$ primitive characters.

\section{Further directions: family averages}\label{sect:family}

We conclude by recording what the preceding calculation suggests about averages over the family of primitive characters modulo $q$. Conjectures~\ref{conj:famA} and~\ref{conj:famB} should be viewed as the uniform family versions of Conjectures~\ref{conj:dirMN} and~\ref{conj:dirNg}. Indeed, for fixed $q$ they follow formally by averaging the individual conjectures over the $\varphi^*(q)$ primitive characters modulo $q$, since the constants $A_q/2\pi$ and $(\varphi(q)/q)/2\pi$ depend on the modulus but not on the particular character. When $q$ is allowed to grow with $T$, the same statements require the corresponding uniformity in the character aspect.

The theorems of this paper already give uniform family lower bounds by summing over primitive characters. Thus, at $k=1$, the family lower bounds capture the proportion
\[
\frac{\beta}{1+\beta},
\qquad
\beta=\frac{\log T}{\log qT},
\]
of the conjectural family main terms, exactly the proportion of the individual estimates. The reason is that the proof treats each $L(s,\chi)$ separately: the mollifier length is restricted by the height interval $[1,T]$, and hence by $\xi\le T^{1-\varepsilon}$, independently of the conductor.

A natural next question is whether a genuinely family-averaged argument can lessen this conductor loss. One would apply the Cauchy--Schwarz inequality once over all pairs $(\chi,\rho)$, with a common mollifier length and polynomial but the usual $\chi$-twisted mollifier for each character. The resulting family second moment would involve sums over primitive characters, and hence the orthogonality relation
\[
\sideset{}{^*}\sum_{\chi\bmod q}\chi(m)\bar\chi(n)=\sum_{d\mid(q,\,m-n)}\mu(q/d)\,\varphi(d)\qquad((mn,q)=1).
\]
Its off-diagonal terms are therefore organised by the congruence conditions $m\equiv n\pmod d$ with $d\mid q$, rather than by equality alone. This extra averaging may allow one to replace the single-character length restriction $\xi\le T^{1-\varepsilon}$ by a family-length restriction of size roughly $\xi\le(qT)^{1-\varepsilon}$.

In the notation of \eqref{eq:master}, such an improvement would let
\[
u=\frac{\log \xi}{\log qT}
\]
approach $1$, rather than only $\beta=\log T/\log qT$, and the Cauchy--Schwarz factor would then satisfy
\[
\frac{u}{1+u}\longrightarrow \tfrac12 .
\]
Thus a direct family average of this type could plausibly lessen the deficit in the conductor aspect, for instance improving the captured proportion from $1/(2+A)$ to $\tfrac12$ when $q=T^{A}$. It would not, however, remove the deficit within the present mollifier framework: the conjectures correspond to the formal long-mollifier limit $u\to\infty$, not merely to $u\to1$.

Making this rigorous would require conductor-uniform analogues of the explicit formula of Landau \cite{Landau} and Gonek \cite{GonekLandau,GonekMean} (see also \cite{DHPC}) for the zero-sums
\[
\sum_{0<\gamma_\chi\le T}\left(\frac{m}{n}\right)^{i\gamma_\chi},
\]
averaged over primitive characters modulo $q$, uniformly in both the conductor and the frequencies $m/n$. Such formulae must control the prime-power contributions and the primitive-character congruences simultaneously, which lies beyond the single-character estimates used here; we do not pursue it.

Finally, the individual $L(s,\chi)$ considered here and the family of primitive characters modulo $q$ are both of unitary type in the sense of Katz and Sarnak \cite{KatzSarnak}. It would be interesting to investigate the corresponding discrete moments in families of other symmetry types, for example the symplectic family of real characters or orthogonal families of modular-form $L$-functions. The results of the present paper apply to each individual member of such a family when the relevant hypotheses are assumed; the genuinely new question is whether averaging over the family permits longer mollifiers, and hence stronger lower bounds, than are available one $L$-function at a time.

\section*{Acknowledgments}
The author gratefully acknowledges support from the Heilbronn Institute for Mathematical Research. Thanks also to Ben Durkan for some helpful comments and references.


\begin{thebibliography}{99}
\bibitem{BFM} H. M. Bui, A. Florea, M. B. Milinovich, \emph{Negative discrete moments of the derivative of the Riemann zeta-function}, Bull. London Math. Soc. \textbf{56} (2024), no.~8, 2680--2703.
\bibitem{CFKRS} J. B. Conrey, D. W. Farmer, J. P. Keating, M. O. Rubinstein, N. C. Snaith, \emph{Integral moments of $L$-functions}, Proc. London Math. Soc. (3) \textbf{91} (2005), 33--104.
\bibitem{CFZ} J. B. Conrey, D. W. Farmer, M. R. Zirnbauer, \emph{Autocorrelation of ratios of $L$-functions}, Comm. Number Theory Phys. \textbf{2} (2008), no.~3, 593--636.
\bibitem{ConreySnaith} J. B. Conrey, N. C. Snaith, \emph{Applications of the $L$-functions ratios conjectures}, Proc. London Math. Soc. (3) \textbf{94} (2007), no.~3, 594--646.
\bibitem{Davenport} H. Davenport, \emph{Multiplicative Number Theory}, 3rd ed., revised by H. L. Montgomery, Graduate Texts in Mathematics \textbf{74}, Springer-Verlag, New York, 2000.
\bibitem{DHPC} B. Durkan, C. P. Hughes, A. Pearce-Crump, \emph{Generalisations of the Landau--Gonek theorem and applications to mean values of zeta}, arXiv:2601.18025 (2026).
\bibitem{CPNT} A. Hamieh, H. Kadiri, G. Martin, N. Ng, \emph{Comparative Prime Number Theory Problem List}, arXiv:2407.03530 (2024).
\bibitem{Farmer} D. W. Farmer, \emph{Long mollifiers of the Riemann zeta-function}, Mathematika \textbf{40} (1993), no.~1, 71--87.
\bibitem{GaoZhao} P. Gao, L. Zhao, \emph{Lower bounds for negative moments of $\zeta'(\rho)$}, Mathematika \textbf{69} (2023), no.~4, 1081--1103.
\bibitem{Gonek} S. M. Gonek, \emph{The second moment of the reciprocal of the Riemann zeta-function and its derivative}, talk at the Mathematical Sciences Research Institute, Berkeley, June 1999.
\bibitem{GonekLandau} S. M. Gonek, \emph{An explicit formula of Landau and its applications to the theory of the zeta-function}, Contemp. Math. \textbf{143} (1993), 395--413.
\bibitem{GonekMean} S. M. Gonek, \emph{Mean values of the Riemann zeta function and its derivatives}, Invent. Math. \textbf{75} (1984), 123--141.
\bibitem{HLZ} W. Heap, J. Li, J. Zhao, \emph{Lower bounds for discrete negative moments of the Riemann zeta function}, Algebra Number Theory \textbf{16} (2022), 1589--1625.
\bibitem{IK} H. Iwaniec, E. Kowalski, \emph{Analytic Number Theory}, American Mathematical Society Colloquium Publications \textbf{53}, American Mathematical Society, Providence, RI, 2004.
\bibitem{KatzSarnak} N. M. Katz, P. Sarnak, \emph{Zeroes of zeta functions and symmetry}, Bull. Amer. Math. Soc. (N.S.) \textbf{36} (1999), no.~1, 1--26.
\bibitem{Landau} E. Landau, \emph{\"Uber die Nullstellen der Zetafunktion}, Math. Ann. \textbf{71} (1912), 548--564.
\bibitem{MilinovichNg} M. B. Milinovich, N. Ng, \emph{A note on a conjecture of Gonek}, Funct. Approx. Comment. Math. \textbf{46.2} (2012), 177--187.
\bibitem{MV} H. L. Montgomery, R. C. Vaughan, \emph{Hilbert's inequality}, J. London Math. Soc. (2) \textbf{8} (1974), 73--82.
\bibitem{NgThesis} N. Ng, \emph{Limiting distributions and zeros of Artin $L$-functions}, PhD thesis, University of British Columbia, 2000.
\bibitem{NgPNET} N. Ng, \emph{Prime number error terms}, preprint, arXiv:2505.11295 (2025).
\bibitem{PCNg} A. Pearce-Crump, \emph{A note on a conjecture of Ng}, arXiv:2606.12376 (2026).
\bibitem{RudnickSound} Z. Rudnick, K. Soundararajan, \emph{Lower bounds for moments of $L$-functions}, Proc. Natl. Acad. Sci. USA \textbf{102} (2005), 6837--6838.
\bibitem{SinhaThesis} S. Sinha, \emph{Some Conditional Results involving Arithmetic Functions}, PhD thesis, University of Missouri, 2025.
\bibitem{Titchmarsh} E. C. Titchmarsh, \emph{The Theory of the Riemann Zeta-function}, 2nd ed., revised by D. R. Heath-Brown, Oxford University Press, 1986.
\bibitem{Tsang} K. M. Tsang, \emph{Some $\Omega$-theorems for the Riemann zeta-function}, Acta Arith. \textbf{46} (1986), no.~4, 369--395.
\end{thebibliography}
\end{document}